\documentclass[10pt]{amsart}

\usepackage{amssymb, amsthm, amsmath}

\usepackage[foot]{amsaddr}

\usepackage[backref=page]{hyperref}
\hypersetup{
	colorlinks=true,
	linkcolor=blue,
	citecolor=blue,
	urlcolor=blue
}

\usepackage[shortlabels]{enumitem}
\usepackage[normalem]{ulem}

\theoremstyle{plain}
\newtheorem{proposition}{Proposition}[section]
\newtheorem{theorem}[proposition]{Theorem}
\newtheorem{lemma}[proposition]{Lemma}

\newtheorem{corollary}[proposition]{Corollary}

\theoremstyle{definition}

\newtheorem{definition}[proposition]{Definition}
\newtheorem{observation}[proposition]{Observation}
\theoremstyle{remark}
\newtheorem{remark}[proposition]{Remark}

\DeclareMathOperator{\Aut}{Aut}

\DeclareMathOperator{\diam}{diam}

\DeclareMathOperator{\GL}{GL}

\DeclareMathOperator{\PGL}{PGL}

\DeclareMathOperator{\End}{End}

\DeclareMathOperator{\Span}{Span}

\DeclareMathOperator{\Homeo}{Homeo}

\DeclareMathOperator{\CH}{ConvHull}

\DeclareMathOperator{\Spanset}{Span}
\DeclareMathOperator{\hil}{d_{\Omega}}
\DeclareMathOperator{\relint}{rel-int}

\DeclareMathOperator{\Bc}{\mathcal{B}}
\DeclareMathOperator{\Cc}{\mathcal{C}}

\DeclareMathOperator{\Gc}{\mathcal{G}}

\DeclareMathOperator{\Lc}{\mathcal{L}}
\DeclareMathOperator{\Nc}{\mathcal{N}}

\DeclareMathOperator{\Pc}{\mathcal{P}}

\DeclareMathOperator{\Sc}{\mathcal{S}}

\DeclareMathOperator{\Uc}{\mathcal{U}}

\DeclareMathOperator{\Vc}{\mathcal{V}}
\DeclareMathOperator{\Xc}{\mathcal{X}}

\DeclareMathOperator{\Hb}{\mathbb{H}}

\DeclareMathOperator{\Nb}{\mathbb{N}}
\DeclareMathOperator{\Pb}{\mathbb{P}}
\DeclareMathOperator{\Rb}{\mathbb{R}}

\DeclareMathOperator{\partiali}{\partial_i}
\DeclareMathOperator{\partialn}{\partial_n}

\DeclareMathOperator{\dist}{d}
\DeclareMathOperator{\Stab}{Stab}
\DeclareMathOperator{\Haus}{Haus}

\newcommand{\abs}[1]{\left|#1\right|}

\newcommand{\wh}[1]{\widehat{#1}}
\newcommand{\ip}[1]{\left\langle #1\right\rangle}

\newcommand{\eqv}[1]{#1/{\sim}}

\DeclareMathOperator{\qcc}{[\partiali \Cc]_{\Xc}}
\DeclareMathOperator{\limset}{\Lc_{\Omega}}
\DeclareMathOperator{\core}{\Cc_{\Omega}}

\begin{document}
\title[Relatively hyperbolic groups in convex real projective geometry]{The structure of relatively hyperbolic groups in convex real projective geometry}
\author{Mitul Islam}\address{Mathematisches Institut, Im Neuenheimer Feld 205, Heidelberg 69120, Germany}
\email{mislam@mathi.uni-heidelberg.de}
\author{Andrew Zimmer}\address{Department of Mathematics, University of Wisconsin-Madison, Madison, WI 53706.}
\email{amzimmer2@wisc.edu}

\date{\today}
\keywords{}
\subjclass[2010]{}

\maketitle 

\begin{abstract}
In this paper we prove a general structure theorem for relatively hyperbolic groups (with arbitrary peripheral subgroups) acting naive convex co-compactly on properly convex domains in real projective space. We also establish a characterization of such groups in terms of the existence of an invariant collection of closed unbounded convex subsets with good isolation properties. This is a real projective analogue of results of Hindawi-Hruska-Kleiner for ${\rm CAT}(0)$ spaces. We also obtain an equivariant homeomorphism between the Bowditch boundary of the group and a quotient of the ideal boundary. 
\end{abstract}

\section{Introduction}

Let $\Hb^d$ denote real hyperbolic $d$-space and recall that a discrete subgroup $\Gamma \leq \mathrm{Isom}(\Hb^d)$ is called \emph{convex co-compact} if there exists a non-empty $\Gamma$-invariant geodesically convex closed subset $\Cc \subset \Hb^d$ where the quotient $\Gamma \backslash \Cc$ is compact. The Beltrami-Klein model realizes $\Hb^d$ as a properly convex domain $\mathbb{B}^d \subset \Pb(\Rb^{d+1})$ (namely the Euclidean unit ball in a standard affine chart) in such a way that the isometry group  $\mathrm{Isom}(\Hb^d)$ coincides with the subgroup of $\PGL_d(\Rb)$ which preserves $\mathbb{B}^d$ (namely ${\rm PSO}(d,1)$ up to conjugation). Further, in this model, a subset being geodesically convex  is equivalent to being convex in some (hence any) affine chart that contains $\mathbb{B}^d$.

The Beltrami-Klein model perspective allows one to naturally generalize the classical notion of convex co-compact groups. In particular, one can consider a general properly convex domain $\Omega \subset \Pb(\Rb^d)$ and the group $\Aut(\Omega) \leq \PGL_d(\Rb)$ of projective automorphisms which preserve $\Omega$. Then a discrete subgroup $\Gamma \leq \Aut(\Omega)$ is called \emph{naive convex co-compact} if there exists a non-empty $\Gamma$-invariant closed convex subset $\Cc \subset \Omega$ where the quotient $\Gamma \backslash \Cc$ is compact. In this case, we say that $(\Omega,\Cc, \Gamma)$ is a \emph{naive convex co-compact triple}. 

When the group $\Gamma$ is word hyperbolic (e.g. in the classical real hyperbolic setting), there is a close connection between this notion of naive convex co-compact groups and Anosov representations/higher Teichm\"{u}ller theory, see \cite{DGF2017} or \cite{Z2017}. As one moves beyond the word hyperbolic case, the structure of these discrete groups becomes more mysterious. 

In this paper we study naive convex co-compact groups which are (intrinsically) relatively hyperbolic groups. This is a rich class with many examples, for instance where
\begin{enumerate}[(a)]
\item $\Gamma$ is a projective reflection group generated by reflection along faces of a projective Coxeter polytope \'{a} la Vinberg and $\Gamma$ is irreducible as a Coxeter group (see \cite{LM2022,DGKLM2021}), and 
\item $\Gamma$ is isomorphic to $\pi_1(M)$ where $M$ is a closed three-manifold such that each geometric component in its JSJ decomposition supports a hyperbolic structure  (see \cite{B2006,IZ2020}). 
\end{enumerate}
We note that in both of these cases, the groups are relatively hyperbolic with respect to peripheral subgroups which are virtually Abelian of rank at least two. However, there also exist examples where $\Gamma$ is relatively hyerpbolic with respect to non virtually Abelian subgroups, see the discussion in~\cite[Section 2.6.3]{W2020}. The primary goal of this paper is to  describe the structure of such examples. One of our main results is proving that the relative hyperbolicity of $\Gamma$ is equivalent to the existence of a so-called \emph{peripheral family}, a $\Gamma$-invariant collection of closed convex subsets with good isolation properties.

This investigation extends some recent work. Previously in~\cite{IZ2019b}, we considered the special case where the peripheral subgroups were virtually Abelian of rank at least two. For such groups, we proved that relative hyperbolicity is equivalent to the existence of a collection of properly embedded simplices with good isolation properties. This geometric description is analogous to the case of ${\rm CAT}(0)$ spaces with isolated flats \cite{HK2005}. We also showed that the boundary of these simplices are, in a technical sense, the only places where the boundary is irregular (see~\cite[Theorem 1.19 (6), Theorem 1.8 (7,8)]{IZ2019b}). Shortly after, Weisman~\cite{W2020} considered convex co-compact groups (a more restrictive class than naive convex co-compact) who were relatively hyperbolic and whose peripheral subgroups were also convex co-compact, but not necessarily virtually Abelian. For such groups he established a similar result about the irregular  boundary points and also showed that the Bowditch boundary could be realized as a quotient of the  boundary.

 There are a number of other results in the literature concerning relatively hyperbolic groups acting on properly convex domains, see for instance~\cite{CLT2015, Choi2017, Choi2017b, Choi_book}. These results consider the case when $\Gamma \backslash \Cc$ is non-compact and characterize when $\Gamma$ is relatively hyperbolic with respect to the fundamental groups of the ends (under some geometric assumptions on the ends and $\Cc$). There is some similarity between the statements in this paper and the statements in~\cite{Choi2017, Choi2017b, Choi_book}, but to the best of our knowledge, there is no actual overlap between the results.

In this paper we extend  the results in~\cite{IZ2019b, W2020} to the case of general peripheral subgroups and the case of general naive convex co-compact groups. The proofs build upon ideas from both papers.

We will now introduce the notation required to precisely state our main results. Given a properly convex domain $\Omega \subset \Pb(\Rb^d)$, we will let $\hil$ denote the Hilbert metric, which is a natural and classical $\Aut(\Omega)$-invariant, proper, and complete metric on $\Omega$, see Section~\ref{sec: convexity and the Hilbert metric} below. This allows us to speak of the diameter $\diam_\Omega(A)$ and $r$-open neighborhood $\Nc_\Omega(A, r) \subset \Omega$ of a subset $A \subset \Omega$ relative to the Hilbert metric.

\begin{definition} 
\label{defn:peripheral_family}
Suppose that $\Omega \subset \Pb(\Rb^d)$ is a properly convex domain, $\Cc \subset \Omega$ is a closed convex subset, $\Gamma \leq \Aut(\Omega)$, and $\Xc$ is a collection of closed unbounded convex subsets of $\Cc$. Then we say
\begin{enumerate}
\item $\Xc$ is $\Gamma$\emph{-invariant} if $\Gamma \cdot \Xc =\Xc$.
\item $\Xc$ is \emph{strongly isolated} if for every $r > 0$ there exists $D_1(r) > 0$ such that: if $X_1, X_2 \in \Xc$ are distinct, then
\begin{align*}
\diam_\Omega \left( \Nc_\Omega(X_1,r) \cap \Nc_\Omega(X_2,r) \right) \leq D_1(r).
\end{align*}
\item $\Xc$ \emph{coarsely contains the properly embedded simplices of $\Cc$} if there exists $D_2 > 0$ such that: if $S \subset \Cc$ is a properly embedded simplex of dimension at least two, then there exists $X \in \Xc$ with $S \subset \Nc_\Omega(X, D_2)$. 
\end{enumerate}
When $(\Omega, \Cc, \Gamma)$ is a naive convex co-compact triple and $\Xc$ satisfies all three of the above conditions we say that $\Xc$ is a  \emph{peripheral family of $(\Omega, \Cc, \Gamma)$}. 
\end{definition}

Given a convex subset $\Cc \subset \Omega$ of a properly convex domain, the \emph{ideal boundary} of $\Cc$ is $\partiali \Cc: = \overline{\Cc} \cap \partial\Omega$. Also, given $x \in \overline{\Omega}$ we will let $F_\Omega(x)$ denote the open face of $x$ in  $\overline{\Omega}$ and given $A \subset \overline{\Omega}$ we will let $F_\Omega(A) = \cup_{x \in A} F_\Omega(x)$. Using these boundary faces and a peripheral family, one can define a natural quotient of the ideal boundary

\begin{definition}\label{defn:boundary quotient} Suppose that $\Xc$ is a peripheral family of $(\Omega, \Cc, \Gamma)$. Let 
$$ 
\qcc: = \partiali\Cc /{\sim}
$$ 
be the topological quotient where  $x \sim y$ if and only if 
\begin{enumerate}
\item $x, y \in F_\Omega(\partiali X)$ for some $X \in \Xc$ or
\item $F_\Omega(y) = F_{\Omega}(x)$.
\end{enumerate}
\end{definition}

\begin{remark} We remark that similar boundary quotients are considered by Choi~\cite{Choi_book} and Weisman~\cite{W2020}. 

\end{remark}

We need one more piece of notation to state our main result: the \emph{limit set} of a subgroup $G \leq \Aut(\Omega)$ is 
$$
\limset(G):=\partial \Omega \cap \bigcup_{p \in \Omega} \overline{G \cdot p}
$$
(unlike real hyperbolic geometry the accumulation points of an orbit may depend on the base point).

\begin{theorem}\label{thm:main} Suppose that $(\Omega, \Cc, \Gamma)$ is a naive convex co-compact triple. Then the following are equivalent: 
\begin{enumerate}
\item If $\Gamma$ is relatively hyperbolic with respect $\Pc=\{P_1, \ldots, P_m\}$ and $X_j$ is the closed convex hull of  $ \limset(P_j) \cap \partiali \Cc$ in $\Omega$, then 
 $$
 \Xc := \Gamma \cdot\{X_1,\dots, X_m\}
 $$ 
 is a peripheral family of $(\Omega, \Cc, \Gamma)$.
\item If $\Xc$ is a peripheral family of $(\Omega, \Cc, \Gamma)$ and $\Pc:=\{P_1, \ldots, P_m\}$ is a set of representatives of the $\Gamma$-conjugacy classes in $\{ \Stab_{\Gamma}(X) : X \in \Xc\}$, then $\Gamma$ is relatively hyperbolic with respect $\Pc$. 
\end{enumerate}
Moreover, when either of the above conditions are satisfied, then:
\begin{enumerate}[(a)]
\item $(\Cc, \dist_\Omega)$ is relatively hyperbolic with respect to $\Xc=\Gamma \cdot \{X_1, \dots, X_m\}$. 
\item There is a $\Gamma$-equivariant homeomorphism between the Bowditch boundary $\partial(\Gamma,\Pc)$ and $\qcc$.
\item Each $(\Omega, X_j, P_j)$ is a naive convex co-compact triple.
\item There exists $L > 0$ such that: if $x \in \partiali\Cc$ and $\diam_{F_\Omega(x)} \left( F_\Omega(x) \cap \partiali\Cc \right)\geq L$, then $x \in F_\Omega(\partiali X)$ for some $X \in \Xc$. 
\item There exists $R > 0$ such that: If $X \in \Xc$ and $x \in \partiali X$, then 
\begin{align*}
\dist_{F_\Omega(x)}^{\Haus}\left( F_\Omega(x) \cap \partiali X, F_\Omega(x) \cap \partiali\Cc\right) \leq R. 
\end{align*}
\end{enumerate} 
\end{theorem}

\begin{remark} Properties (d) and (e) are somewhat technical. Informally, property (d) states that any boundary face of $\Omega$ that $\partiali\Cc$ intersects in a ``large set'' must intersect the ideal boundary of an element in $\Xc$. Property (e) informally states that if $X \in \Xc$, then $\partiali X$ coarsely contains any boundary face of $\partiali\Cc$ that it intersects. 
\end{remark}

The equivalence part of Theorem~\ref{thm:main} can be viewed as a real projective analogue of results of Hruska--Kleiner and Hindawi--Hruska--Kleiner~\cite{HK2005,HK2009} in the setting of ${\rm CAT}(0)$-geometry. This earlier work motivated the results in this paper, but the methods of proof are very different. We should also note that an old result of Kelly-Strauss~\cite{KS1958} says that a Hilbert geometry $(\Omega, \dist_\Omega)$ is ${\rm CAT}(0)$ if and only if it is isometric to real hyperbolic $(d-1)$-space (in which case $\Omega$ coincides, up to a change of coordinates, with the Beltrami-Klein model of real hyperbolic $(d-1)$-space).

As mentioned above, in previous work~\cite{IZ2019b} we proved a version of Theorem~\ref{thm:main} in the special case when $\Xc$ consisted of properly embedded simplices of dimension at least two and the subgroups in $\Pc$ were virtually Abelian of rank at least two. In this case, using the simple structure of simplices and Abelian subgroups in naive convex co-compact groups (see~\cite{IZ2019}), one can weaken the strongly isolated assumption to only assuming that $\Xc$ is closed and discrete. In fact, a substantial portion of this earlier work was building a strongly isolated collection of properly embedded simplices from a closed and discrete collection. It seems unlikely to us that such a weakening is possible in the general case.

\subsection{Convex co-compact groups} As mentioned above, the class of naive convex co-compact groups includes the more restrictive, but still interesting, class of convex co-compact groups. For this class of groups, Theorem~\ref{thm:main} can be restated in a much simpler way.

Given a properly convex domain $\Omega \subset \Pb(\Rb^d)$ and a discrete subgroup $\Gamma \leq \Aut(\Omega)$, let $\core(\Gamma) \subset \Omega$ denote the closed convex hull of $\limset(\Gamma)$ in $\Omega$. 

\begin{definition}
\label{defn:cc}
A discrete subgroup $\Gamma \leq \Aut(\Omega)$ is called \emph{convex co-compact} if $\core(\Gamma)$ is non-empty and the quotient $\Gamma \backslash \core(\Gamma)$ is compact. 
\end{definition}

Every convex co-compact subgroup is clearly naive convex co-compact, but the converse is not true (see~\cite[Section 2.6]{IZ2019b} for some examples). One key difference between the two definitions is that if $\Gamma \leq \Aut(\Omega)$ is convex co-compact, then any open boundary face of $\Omega$ which intersects $\partiali \core(\Gamma)$ is actually contained in $\partiali \core(\Gamma)$. On the other hand, if $(\Omega, \Cc, \Gamma)$ is a naive convex co-compact triple, then it is possible for $\partiali\Cc$ to intersect a boundary face of $\Omega$ in a small set. This seemingly small difference makes convex co-compact groups much easier to study. 

For convex co-compact groups, the ``moreover'' part of Theorem~\ref{thm:main} can be restated and expanded as follows. 

\begin{theorem}[see Section \ref{sec:cc}]
\label{thm:main_cc}
Suppose that $\Omega \subset \Pb(\Rb^d)$ is a properly convex domain and $\Gamma \leq \Aut(\Omega)$ is a convex co-compact subgroup. If $\Gamma$ is relatively hyperbolic with respect to $\Pc=\{P_1,\dots, P_m\}$, then:
\begin{enumerate}[(a)]
\item $(\core(\Gamma), \dist_\Omega)$ is relatively hyperbolic with respect to 
$$
\Xc:=\Gamma \cdot \{\core(P_1),\dots, \core(P_m)\}.
$$
\item Let $[\partiali \core(\Gamma)]_{\Pc}$ denote the quotient of $\partiali \core(\Gamma)$ obtained by collapsing each limit set $\limset(\gamma P_j \gamma^{-1})$ to a point (where $\gamma \in \Gamma$ and $P_j \in\Pc$). Then there is a $\Gamma$-equivariant homeomorphism  between the Bowditch boundary $\partial(\Gamma,\Pc)$ and $[\partiali\core(\Gamma)]_{\Pc}$. 

\item Each $P_j$ is a convex co-compact subgroup of $\Aut(\Omega)$.
\item If $x \in \partiali\core(\Gamma)$ is not a $\Cc^1$-smooth point of $\Omega$ (i.e. $\Omega$ does not have a unique supporting hyperplane at $x$), then $x \in \Lc_\Omega(\gamma P_j\gamma^{-1})$ for some $\gamma \in \Gamma$ and $P_j \in\Pc$.
\item If $\ell \subset \partiali\core(\Gamma)$ is a non-trivial line segment, then $\ell \subset\Lc_\Omega( \gamma P_j \gamma^{-1})$ for some $\gamma \in \Gamma$ and $P_j \in\Pc$.
\end{enumerate}
\end{theorem}

We should note that the main new content of Theorem~\ref{thm:main_cc} is part (c). In particular: 
\begin{enumerate}
\item part (a) is an immediate consequence of part (c), 
\item in~\cite{IZ2019b} we previously proved parts (d) and (e) in the case where each $P_j$ is virtually Abelian with rank at least two and once part (c) is known the same argument works in the more general setting, and
\item Weisman~\cite{W2020} established parts (b) and (e) with part (c) as an assumption.
\end{enumerate}

In the context of convex co-compact groups, peripheral families can also be defined in terms of their boundary behavior. 

\begin{proposition}[see Section~\ref{sec:cc}]\label{prop: peripheral family in cc case} Suppose that $\Omega \subset \Pb(\Rb^d)$ is a properly convex domain and $\Gamma \leq \Aut(\Omega)$ is a convex co-compact subgroup. If $\Xc$ is a $\Gamma$-invariant collection of closed unbounded convex subsets of $\Omega$, then the following are equivalent: 
\begin{enumerate}
\item $\Xc$ is a peripheral family of $(\Omega, \core(\Gamma), \Gamma)$.
\item $\Xc$ has the following properties: 
\begin{enumerate}
\item $\Xc$ is closed in the local Hausdorff topology induced by the Hilbert metric $\dist_\Omega$. 
\item If $X_1, X_2 \in \Xc$ are distinct, then $\partiali X_1 \cap \partiali X_2 = \emptyset$. 
\item If $\ell \subset \partiali\core(\Gamma)$ is a non-trivial line segment, then $\ell \subset \partiali X$ for some $X \in \Xc$. 
\end{enumerate} 
\end{enumerate}
Moreover, when the above conditions are satisfied, then 
$$
[\partiali \core(\Gamma)]_{\Pc} = [\partiali \core(\Gamma)]_{\Xc}
$$
where $\Pc$ is a set of representatives of the $\Gamma$-conjugacy classes in $\{ \Stab_{\Gamma}(X) : X \in \Xc\}$.
\end{proposition}  

Previously, Weisman~\cite[Theorem 1.16]{W2020} characterized when a convex co-compact subgroup is hyperbolic relative to a collection of convex co-compact subgroups in terms of the behavior of the limit sets of the subgroups. We note that combining Proposition~\ref{prop: peripheral family in cc case} and Theorem~\ref{thm:main} provides a ``subset''  version of this characterization.

\subsection{The case of Gromov hyperbolic groups} A word hyperbolic group is relatively hyperbolic with respect to the empty set. Thus if $(\Omega,\Cc,\Gamma)$ is a naive convex co-compact triple and $\Gamma$ is a  word hyperbolic group, then Theorem \ref{thm:main} holds with $\Xc=\emptyset$. 

However, the choice of this peripheral family $\Xc$ is not canonical even for a convex co-compact group $\Gamma$ (unlike our work in \cite{IZ2019b} where $\Xc$ was the set of all maximal properly embedded simplices in $\core(\Gamma)$ of dimension at least two). For instance, suppose that $\Gamma:=\ip{a, b}$ is a convex co-compact  subgroup of $\mathrm{PSO}(2,1)$ that is isomorphic to a free group on two generators. Then $\Gamma$ is relatively hyperbolic with respect to $\Pc:=\{\langle g \rangle \}$ where $g:=aba^{-1}b^{-1}$. Then $\Xc:=\Gamma \cdot (g^+,g^-)$ where $(g^+,g^-)$ is the unique $g$-invariant projective line in the Beltrami-Klein model of $\Hb^2$. On the other hand, we can also choose $\Xc=\emptyset$ (when $\Pc=\emptyset$).

\subsection{Outline of the paper and proofs} 

Sections~\ref{sec: preliminaries} and~\ref{sec:rel hyp groups} are expository. In Section~\ref{sec: preliminaries} we recall the basic definitions and results about properly convex domains that we will require and in Section~\ref{sec:rel hyp groups} we recall some properties of relatively hyperbolic spaces and groups. 

Sections~\ref{sec:finding properly embedded simplices} through~\ref{sec:proof_main} are devoted to the proof of Theorem~\ref{thm:main}. The difficult direction of the equivalence is showing that (2) implies (1). Our general strategy is based on combining ideas from~\cite{IZ2019b} and~\cite{W2020}. In particular, as in~\cite{W2020} we will use Yaman's characterization of relatively hyperbolicity to show that $\Gamma$ is relatively hyperbolic. We will verify the conditions in Yaman's theorem by further developing the ideas used in~\cite{IZ2019b} to study closed and discrete collections of properly embedded simplices.

We should also note that Choi~\cite{Choi_book} (which predates the work of Weisman) used Yaman's characterization to verify that the fundamental groups of certain non-compact convex real projective manifolds were relatively hyperbolic with respect to the fundamental groups of their ends.

The arguments used to verify the conditions of Yaman's theorem in this paper are somewhat similar to the analogous ones in~\cite{W2020, Choi_book}, but given the different setups, there doesn't seem to be an easy to way to reduce the proofs in this paper to any lemmas in the work of~\cite{W2020, Choi_book}.

The content of these sections is as follows:
\begin{enumerate}
\item In Section~\ref{sec:finding properly embedded simplices} we prove a quantitative result which informally states that long line segments in an open boundary face (relative to the Hilbert metric of the face) imply the existence of nearby properly embedded simplices. 
\item In Section~\ref{sec:properties of peripheral families} we establish a number of useful properties of peripheral families. 
\item In Section~\ref{sec:peripheral subgroups are ncc} we prove that  peripheral subgroups of a relatively hyperbolic naive convex co-compact subgroup are themselves naive convex co-compact. 
\item In Section~\ref{sec: co-compact boundary actions} we prove a technical result which will allow us to show that non-conical limit points are bounded parabolic points. A key tool here is a ``closest subset'' projection map.
\item In Section~\ref{sec:bdry_quotient} we prove some basic properties of the quotient space $\qcc$, including a sufficient condition for a point to be a conical limit point. 
\item In Section \ref{sec:proof_main} we put everything together and prove Theorem \ref{thm:main}.
\end{enumerate} 
Finally, in Section~\ref{sec:cc}, we use Theorem~\ref{thm:main} to prove Theorem~\ref{thm:main_cc}.

\subsection*{Acknowledgements} M. Islam thanks Louisiana State University for hospitality during a visit in March 2020 where work on this project started. 

M. Islam was partially supported by Emmy Noether Project 427903332 (funded by the DFG), DMS-1607260 (National Science Foundation), and Rackham Predoctoral Fellowship (University of Michigan). A. Zimmer was partially supported by grants DMS-2105580 and DMS-2104381 from the National Science Foundation.

\section{Preliminaries}\label{sec: preliminaries}

\subsection{Notation} If $V \subset \Rb^d$ is a non-zero linear subspace, we will let $\Pb(V) \subset \Pb(\Rb^d)$ denote its projectivization. In most other cases, we will use $[o]$ to denote the projective equivalence class of an object $o$, for instance: 
\begin{enumerate}
\item if $v \in \Rb^{d} \setminus \{0\}$, then $[v]$ denotes the image of $v$ in $\Pb(\Rb^{d})$, 
\item if $\phi \in \GL_{d}(\Rb)$, then $[\phi]$ denotes the image of $\phi$ in $\PGL_{d}(\Rb)$, and 
\item if $T \in \End(\Rb^{d}) \setminus\{0\}$, then $[T]$ denotes the image of $T$ in $\Pb(\End(\Rb^{d}))$. 
\end{enumerate}

We also standardize some metric notations. If $(X,\dist)$ is a metric space, $p\in X$, $A \subset X$, and $r>0$, then
\begin{enumerate}
\item $\Nc_X(A,r):=\{x \in X : \dist(x,A)<r\}$,
\item $\Bc_X(p,r):=\{x \in X : \dist(p,x)< r\}$, and 
\item $\diam_X(A):=\sup \{ \dist(x,x'): x, x' \in A\}.$
\end{enumerate}

\subsection{Convexity and the Hilbert metric}\label{sec: convexity and the Hilbert metric} A subset $C\subset \Pb(\Rb^d)$ is:
\begin{enumerate}
\item \emph{properly convex} if there exists an affine chart $\mathbb{A}$ of $\Pb(\Rb^d)$ where $C \subset \mathbb{A}$ is a bounded convex subset.
\item  a \emph{properly convex domain}  if $C$ is properly convex and open in $\Pb(\Rb^d)$. 
\end{enumerate}
Given a properly convex set $C \subset \Pb(\Rb^d)$ and a subset $X \subset \overline{C}$, we define its \emph{convex hull} as 
\begin{align*}
{\rm ConvHull}_C(X):= \cap \big\{ Y : Y  \text{ is a closed convex subset such that } X \subset Y \subset \overline{C}\big\}.
\end{align*}

A \emph{line segment} in $\Pb(\Rb^{d})$ is a connected subset of a projective line. Given two points $x,y \in \Pb(\Rb^{d})$ there is no canonical line segment with endpoints $x$ and $y$, but we will use the following convention: if $C \subset \Pb(\Rb^d)$ is a properly convex set and $x,y \in \overline{C}$, then (when the context is clear) we will let $[x,y]$ denote the closed line segment joining $x$ to $y$ which is contained in $\overline{C}$. In this case, we will also let $(x,y)=[x,y]\setminus\{x,y\}$, $[x,y)=[x,y]\setminus\{y\}$, and $(x,y]=[x,y]\setminus\{x\}$. 
 
Suppose that $\Omega \subset \Pb(\Rb^d)$ is a properly convex domain. If $x, y \in \Omega$ are distinct, let $[a,b]:=\Pb(\Span\{x,y\}) \cap \overline{\Omega}$ where $a$ and $b$ labelled such that $x \in [a,y]$ (i.e. the points are ordered $a, x, y, b$ along $[a,b]$). Then the \emph{the Hilbert distance} between $x$ and $y$ is defined to be
\begin{align*}
\hil(x,y) := \frac{1}{2}\log [a, x,y, b]
\end{align*}
 where 
 \begin{align*}
 [a,x,y,b] := \frac{\abs{x-b}\abs{y-a}}{\abs{x-a}\abs{y-b}}
 \end{align*}
 is the cross ratio (here $\abs{\cdot}$ is some (any) norm in some (any) affine chart which contains $a,x,y,b$). Then $(\Omega, \hil)$ is a complete geodesic metric space and $\Aut(\Omega)$ acts properly and by isometries on $\Omega$ (see for instance~\cite[Section 28]{BK1953}). Further, the projective line segment $[x,y]$ is a geodesic for the Hilbert distance.

\subsection{Boundaries and faces}

Suppose that $C \subset \Pb(\Rb^d)$ is a properly convex set. The \emph{relative interior of $C$}, denoted by $\relint(C)$, is  the interior of $C$ in $\Pb(\Spanset C)$. We will say that $C$ is  \emph{open in its span} if $C = \relint(C)$. 

The \emph{boundary of $C$} is $\partial C : = \overline{C} \setminus \relint(C)$, the \emph{ideal boundary of $C$} is $\partiali C := \partial C \setminus C$, and the \emph{non-ideal boundary of $C$} is $\partialn C := \partial C \cap C$.

\begin{definition}\label{defn:open_faces}
If $C \subset \Pb(\Rb^d)$ is a properly convex set which is open in its span  and $x \in \overline{C}$, let $F_C(x)$ denote the \emph{open face} of $x$, that is 
\begin{align*}
F_C(x) = \{ x\} \cup \left\{ y \in \overline{C} :  \exists \text{ an open line segment in } \overline{C} \text{ containing } x \text{ and }y \right\}.
\end{align*}
\end{definition}

Directly from the definitions we have the following. 

\begin{observation}\label{obs:faces} Suppose that $\Omega \subset \Pb(\Rb^d)$ is a properly convex domain. 
\begin{enumerate}
\item $F_\Omega(x) = \Omega$ when $x \in \Omega$,
\item $F_\Omega(x)$ is open in its span,
\item $y \in F_\Omega(x)$ if and only if $x \in F_\Omega(y)$ if and only if $F_\Omega(x) = F_\Omega(y)$,
\item if $x, y \in \overline{\Omega}$, $z \in (x,y)$, $p \in F_{\Omega}(x)$, and $q \in F_{\Omega}(y)$, then 
\begin{align*}
(p,q) \subset F_\Omega(z).
\end{align*}
In particular, $(p,q) \subset \Omega$ if and only if $(x,y) \subset \Omega$.
\end{enumerate}
\end{observation}

If $B \subset C \subset \Pb(\Rb^d)$ are properly convex sets, then we say that $B$ is \emph{properly embedded} in $C$ if $B \hookrightarrow C$ is a proper map with respect to the subspace topology. Note that $B$ is properly embedded in $C$ if and only if $\partiali B \subset \partiali C$. 

\subsection{Limits of automorphisms} Every $T \in \Pb(\End(\Rb^d))$ induces a map 
\begin{align*}
\Pb(\Rb^d) \setminus \Pb(\ker T) \rightarrow \Pb(\Rb^d)
\end{align*}
defined by $x \mapsto T(x)$. We will frequently use the following observation.

\begin{observation}\label{obs:limits_of_maps} If $\{T_n\}$ is a sequence in  $\Pb(\End(\Rb^d))$ converging to $T \in \Pb(\End(\Rb^d))$, then 
\begin{align*}
T(x) = \lim_{n \rightarrow \infty} T_n(x)
\end{align*}
for all $x \in \Pb(\Rb^d) \setminus \Pb(\ker T)$. Moreover, the convergence is uniform on compact subsets of $ \Pb(\Rb^d) \setminus \Pb(\ker T)$. 
\end{observation}

We can view $\Pb(\End(\Rb^d))$ as a compactification of $\PGL_d(\Rb)$ and then consider limits of automorphisms in this compactification. 

\begin{proposition}\cite[Proposition 5.6]{IZ2019}\label{prop:dynamics_of_automorphisms_1}
Suppose that $\Omega \subset \Pb(\Rb^d)$ is a properly convex domain, $p_0 \in \Omega$, and $\{g_n\}$ is a sequence in $\Aut(\Omega)$ such that 
\begin{enumerate}
\item $g_n (p_0) \rightarrow x \in \partial \Omega$, 
\item $g_n^{-1} (p_0) \rightarrow y \in \partial \Omega$, and
\item $g_n \rightarrow T$ in $\Pb(\End(\Rb^d))$. 
\end{enumerate}
Then ${\rm image}(T) \subset \Spanset  F_\Omega(x)$, $\Pb(\ker T) \cap \Omega = \emptyset$, and $y \in \Pb(\ker T)$. 
\end{proposition} 

In the case of ``non-tangential'' convergence we can say more. 

\begin{proposition}\cite[Proposition 5.7]{IZ2019}\label{prop:dynamics_of_automorphisms_2}
Suppose that $\Omega \subset \Pb(\Rb^d)$ is a properly convex domain, $p_0 \in \Omega$, $x \in \partial \Omega$, $\{p_n\}$ is a sequence in $[p_0, x)$ converging to $x$, and $\{g_n\}$ is a sequence in $\Aut(\Omega)$ such that 
\begin{align*}
\sup_{n \geq 1} \hil(g_n (p_0), p_n) < + \infty.
\end{align*}
If $g_n \rightarrow T \in \Pb(\End(\Rb^d))$, then $T(\Omega) = F_\Omega(x)$. 
\end{proposition}

Proposition 5.7 in~\cite{IZ2019} is stated differently and a proof of the statement above can be found in~\cite[Proposition 2.13]{Z2020}.

\subsection{Projective simplices}

A subset $S \subset \Pb(\Rb^d)$ is called a \emph{$k$-dimensional simplex} in $\Pb(\Rb^d)$  if there exists $g \in \PGL_d(\Rb)$ such that 
\begin{align*}
g S = \left\{ [x_1:\dots:x_{k+1}:0:\dots:0] \in\Pb(\Rb^d) : x_1,\dots,x_{k+1} > 0 \right\}.
\end{align*}
In this case, we call the $k+1$ points 
\begin{align*}
g^{-1}[1:0:\dots:0], g^{-1}[0:1:0:\dots:0], \dots, g^{-1}[0:\dots:0:1:0:\dots:0] \in \partial S
\end{align*}
the \emph{vertices of} $S$. 

The Hilbert metric on a simplex can be explicitly computed (see~\cite[Proposition 1.7]{N1988}, ~\cite{dlH1993}, or ~\cite{V2014}) and from this explicit form one obtains the following.

\begin{proposition}\label{prop:simplex_quasi_isometric_to_R} If $\Omega \subset \Pb(\Rb^d)$ is a properly convex domain and $S \subset \Omega$ is a properly embedded simplex, then $(S, \dist_\Omega)$ is quasi-isometric to $\Rb^{\dim S}$. 
\end{proposition} 

\begin{remark} By definition $\dist_S= \dist_\Omega|_{S \times S}$ and so the quasi-isometry constants only depend on $\dim S$. \end{remark}

\subsection{The Hausdorff distance and local Hausdorff topology}When $(X,\dist)$ is a metric space, the \emph{Hausdorff distance} between two subsets $A,B \subset X$ is defined by 
\begin{align*}
\dist^{\Haus}_X(A,B) = \max \left\{ \sup_{a \in A} \inf_{b \in B} \dist(a,b), \ \sup_{b \in B} \inf_{a \in A} \dist(a,b) \right\}.
\end{align*}
When $(X,\dist)$ is a complete metric space space,  $d^{\Haus}_X$ is a complete metric on the set of non-empty compact subsets of $X$. 

The local Hausdorff topology is a natural topology on the set of non-empty closed sets in $X$. For a non-empty closed set $C_0$, a base point $x_0 \in X$, and $r_0, \epsilon_0 > 0$ define the set $U(C_0,x_0,r_0,\epsilon_0)$ to consist of all closed subsets $C \subset X$ where
\begin{align*}
\dist^{\Haus}_X\Big(C_0 \cap \Bc_X(x_0,r_0),\, C \cap \Bc_X(x_0,r_0)\Big) < \epsilon_0.
\end{align*}
The \emph{local Hausdorff topology} on the set of non-empty closed subsets of $X$ is the topology generated by the sets $U(\cdot, \cdot, \cdot, \cdot)$.

\subsection{Distance estimates for the Hilbert metric}

The asymptotic behavior of the Hilbert distance connects naturally with the structure of open faces in the boundary.

\begin{proposition}\label{prop:asymp_sequences} Suppose that $\Omega \subset \Pb(\Rb^d)$ is a properly convex domain and $\{p_n\}, \{q_n\}$ are sequences in $\Omega$ where $p_n \rightarrow p \in \overline{\Omega}$ and $q_n \rightarrow q \in \overline{\Omega}$. If 
$$
\liminf_{n \rightarrow \infty} \dist_\Omega(p_n, q_n) < +\infty,
$$
then $F_\Omega(p)=F_\Omega(q)$ and 
$$
\dist_{F_\Omega(p)}(p,q) \leq \liminf_{n \rightarrow \infty} \dist_\Omega(p_n, q_n). 
$$
\end{proposition}

\begin{proof} This follows immediately from the definition of the Hilbert metric. \end{proof} 

We will frequently use the following fact about the Hausdorff distance between line segments. 

\begin{proposition}[{\cite[Proposition 5.3]{IZ2019}}]
\label{prop:Hausdorff_distance_between_lines}
 Suppose that $\Omega \subset \Pb(\Rb^d)$ is a properly convex domain. If $p_1, p_2, q_1, q_2 \in \overline{\Omega}$ are such that  $(p_1,p_2) \subset \Omega$, $F_{\Omega}(p_1)=F_{\Omega}(q_1)$, and $F_{\Omega}(p_2)=F_{\Omega}(q_2)$, then 
$$
\dist_\Omega^{\Haus}( [p_1,p_2] \cap \Omega, [q_1, q_2] \cap \Omega ) \leq \max\{ \dist_{F_{\Omega}(p_1)}(p_1,q_1), \dist_{F_{\Omega}(p_2)}(p_2,q_2) \}.
$$
\end{proposition}

Using induction Proposition~\ref{prop:Hausdorff_distance_between_lines} can be upgraded as follows. 

\begin{proposition}\label{prop:Hausdorff_distance_between_ch} Suppose that $\Omega \subset \Pb(\Rb^d)$ is a properly convex domain, $q_1,\dots, q_m \in \overline{\Omega}$, and 
$$
z \in \relint\Big({\rm ConvHull}_\Omega(q_1,\dots,q_m) \Big).
$$
If $p_1, \dots, p_m \in \overline{\Omega}$ and $F_\Omega(p_j) = F_\Omega(q_j)$ for all $1 \leq j \leq m$, then 
$$
\dist_{F}^{\Haus}\Big( {\rm ConvHull}_\Omega(q_1,\dots,q_m) \cap F, {\rm ConvHull}_\Omega(p_1,\dots,p_m) \cap F \Big) \leq \max_{1 \leq j \leq m} \dist_{F_\Omega(q_j)}(p_j,q_j)
$$
where $F:=F_\Omega(z)$. 
\end{proposition} 

\begin{proof} For $1 \leq k \leq m$, let 
$$
S_k = \relint\Big({\rm ConvHull}_\Omega(q_1,\dots,q_k) \Big) \quad \text{and} \quad S^\prime_k = \relint\Big({\rm ConvHull}_\Omega(p_1,\dots,p_k) \Big).
$$
We claim that for each $k$ there exists a face $F_k$ of $\Omega$ such that $S_k \cup S_k^\prime \subset F_k$ and 
$$
\dist^{\Haus}_{F_k}(S_k, S_k^\prime) \leq \max_{1 \leq j \leq k} \dist_{F_\Omega(q_j)}(p_j,q_j).
$$

We induct on $k$. The base case $k=1$ is by definition. Fix $k > 1$ and $w \in S_k$. Then there exists $q^\prime \in S_{k-1}$ such that $w \in (q^\prime,q_k)$. By induction there exists $p^\prime \in S_{k-1}^\prime$ such that 
$$
\dist_{F_{k-1}}(q^\prime, p^\prime) \leq \max_{1 \leq j \leq k-1} \dist_{F_\Omega(q_j)}(p_j,q_j).
$$
Notice that $(p^\prime,p_k) \subset S_k^\prime$ and so Observation \ref{obs:faces} part (4) implies that 
$$
S_k \cup S_k^\prime \subset F_\Omega(w) =: F_k.
$$
Further, by Proposition \ref{prop:Hausdorff_distance_between_lines}, 
$$
\dist_{F_k}\left(w, S_k^\prime\right) \leq \dist_{F_k}^{\Haus}\left( (q^\prime,q_k),(p^\prime,p_k)\right) \leq \max_{1 \leq j \leq k} \dist_{F_\Omega(q_j)}(p_j,q_j).
$$
Since $w \in S_k$ was arbitrary, we see that 
$$
\sup_{w \in S_k} \dist_\Omega\left(w,S_k^\prime\right) \leq \max_{1 \leq j \leq k} \dist_{F_\Omega(q_j)}(p_j,q_j).
$$

Repeating the same argument with $w^\prime \in S_k^\prime$ shows that 
$$
\sup_{w^\prime \in S_k^\prime} \dist_\Omega\left(w^\prime,S_k\right) \leq \max_{1 \leq j \leq k} \dist_{F_\Omega(q_j)}(p_j,q_j).
$$
Hence 
$$
\dist^{\Haus}_{F_k}(S_k, S_k^\prime) \leq \max_{1 \leq j \leq k} \dist_{F_\Omega(q_j)}(p_j,q_j).
$$
This completes the proof of the claim. 

Finally, notice that $F=F_m$ and $S_m$ (respectively $S_m^\prime$) is dense in 
$$
{\rm ConvHull}_\Omega(q_1,\dots,q_m) \cap F
$$ 
(respectively ${\rm ConvHull}_\Omega(p_1,\dots,p_m) \cap F$). So the result follows. 
\end{proof}

\section{Relatively hyperbolic spaces and groups}\label{sec:rel hyp groups}

In this expository section we recall some basic properties of relatively hyperbolic groups and spaces. We define relative hyperbolic spaces and groups in terms of Dru{\c t}u and Sapir's tree-graded spaces (see~\cite[Definition 2.1]{DS2005}). 

\begin{definition} \ \begin{enumerate}
\item A complete geodesic metric space $(X,\dist)$ is \emph{relatively hyperbolic with respect to a collection of subsets $\Sc$} if all its asymptotic cones, with respect to a fixed non-principal ultrafilter, are tree-graded with respect to the collection of ultralimits of the elements of $\Sc$. 
\item A finitely generated group $G$ is \emph{relatively hyperbolic with respect to a family of subgroups $\{H_1,\dots, H_k\}$} if the Cayley graph of $G$ with respect to some (hence any) finite set of generators is relatively hyperbolic with respect to the collection of left cosets $\{g H_i : g \in G, i=1,\dots,k\}$. 
\end{enumerate}
\end{definition}

\begin{remark}These are one among several equivalent definitions of relatively hyperbolic spaces/groups, see~\cite{DS2005} and the references therein for more details. 

\end{remark}

Recall that if $(X,\dist)$ is a metric space, $A \subset X$, and $r> 0$, then 
\begin{align*}
\Nc_X(A,r):=\left\{ x \in X : \dist(x,A)< r\right\}.
\end{align*}
We will frequently use the following properties of relatively hyperbolic spaces.

\begin{theorem}
\label{thm:rh_ds}
Suppose that $(X, \dist_X)$ is relatively hyperbolic with respect to $\Sc$. Then:
\begin{enumerate}

\item \label{thm:rh_embeddings_of_flats} \emph{(Dru{\c t}u-Sapir~\cite[Corollary 5.8]{DS2005})} For any $A \geq 1$ and $B \geq 0$ there exists $M =M(A,B)$ such that: if $k \geq 2$ and $f : \Rb^k \rightarrow X$ is an $(A,B)$-quasi-isometric embedding, then there exists some $S \in \Sc$ such that 
\begin{align*}
f(\Rb^k) \subset \Nc_X(S,M).
\end{align*}

\item \label{thm:quasi-convexity} \emph{(Dru{\c t}u-Sapir~\cite[Lemma 4.15]{DS2005})} For any $A \geq 1$ and $B \geq 0$ there exists $t =t(A,B)$ such that: if  $S \in \Sc$, $r \geq 1$, and $\gamma: [0,T] \rightarrow X$ is an $(A,B)$-quasi-geodesic with $\gamma(0), \gamma(T) \in \Nc_X(S,r)$, then 
\begin{align*}
\gamma([0,T]) \subset \Nc_{X}(S,tr).
\end{align*}

\item \label{thm:abstract_coarsely_isolated} \emph{(Dru{\c t}u-Sapir~\cite[Theorem 4.1]{DS2005})}  For any $r>0$ there exists $D=D(r) > 0$ such that: if  $S_1,S_2 \in \Sc$ are distinct, then 
\begin{align*}
\diam_X \left( \Nc_X(S_1,r) \cap \Nc_X(S_2, r) \right) \leq D.
\end{align*}

\end{enumerate}
\end{theorem}

\subsection{Immediate consequences in convex projective geometry} As an immediate consequence of the general theory of relatively hyperbolic spaces we have the following.

\begin{proposition} 
\label{prop:immediate_consequences_of_DS05}
Suppose that $(\Omega, \Cc, \Gamma)$ is a naive convex co-compact triple and $\Xc$ is a $\Gamma$-invariant collection of closed unbounded convex subsets of $\Cc$. If $(\Cc, \dist_\Omega)$ is relatively hyperbolic with respect to $\Xc$, then $\Xc$ is a peripheral family of $(\Omega, \Cc, \Gamma)$ (i.e. $\Gamma$-invariant, strongly isolated, and coarsely contains the properly embedded simplices of $\Cc$).
\end{proposition}

\begin{proof} Theorem \ref{thm:rh_ds} part \ref{thm:abstract_coarsely_isolated} implies that $\Xc$ is strongly isolated. Proposition~\ref{prop:simplex_quasi_isometric_to_R} and Theorem \ref{thm:rh_ds} part \ref{thm:rh_embeddings_of_flats} imply that $\Xc$  coarsely contains the properly embedded simplices of $\Cc$.
\end{proof}

\subsection{Yaman's characterization}
\label{sec:yaman_creiteria}
 In this section we recall Yaman's characterization of relatively hyperbolic groups in terms of convergence group actions \cite{PT1998, Y2004}. 

Recall that when $M$ is a non-empty compact Hausdorff metrizable space, a subgroup $G \leq \Homeo(M)$ is called a \emph{convergence group} if for every sequence $\{g_n\}$ of distinct elements in $G$ there exists a subsequence $\{g_{n_j}\}$ and points $x,y \in M$ such that $g_{n_j}|_{M - \{y\}}$ converges locally uniformly to the constant map $x$. For such a subgroup:
\begin{enumerate}
\item The \emph{limit set} $\Lc(G) \subset M$ is the set of points $x \in M$ where there exist $y \in M$ and a sequence $\{g_n\}$ in $G$ where $g_n|_{M-\{y\}}$ converges locally uniformly to the constant map $x$.
\item A point $x \in \Lc(G)$ is a \emph{conical limit point} if there exist distinct points $a,b \in M$ and a sequence of elements $\{ g_n\}$ in $G$ where $\lim_{n \rightarrow \infty} g_n(x) = a$ and $\lim_{n \rightarrow \infty} g_n(y) = b$ for all $y \in M -\{x\}$. 
\item A point $x \in \Lc(G)$ is a \emph{parabolic point} if $\Stab_G(x)$ is infinite and non-loxodromic (i.e. $x$ is the unique fixed point in $\Lc(G)$ for every $g \in \Stab_G(x)$ of infinite order).
\item A parabolic point $x \in \Lc(G)$ is a \emph{bounded parabolic point} if $\Stab_G(x)$ acts co-compactly on $\Lc(G)-\{x\}$. 
\end{enumerate}

Finally, we say that a convergence group $G \leq \Homeo(M)$ is \emph{geometrically finite} if $\Lc(G)$ is a non-empty perfect set (i.e. $\#\Lc(G) \geq 3$) and every point in $\Lc(G)$ is either a conical limit point or a bounded parabolic point.  

\begin{theorem}[Yaman~{\cite[Theorem 0.1]{Y2004}}] 
\label{thm:yaman_criteria}
Suppose that $G \leq {\rm Homeo}(M)$ is a geometrically finite convergence group and $B \subset \Lc(G)$ is the set of bounded parabolic fixed points. If $B$ has finitely many $G$-orbits and ${\rm Stab}_G(b):=\{h \in G : h(b) = b\}$ is finitely generated for every $b \in B$, then:
\begin{enumerate}
\item $G$ is relatively hyperbolic with respect to $\Pc=\{\Stab_G(p_1),\dots, \Stab_G(p_m)\}$ where $B=\sqcup_{i=1}^mG(p_i)$, and 
\item the Bowditch boundary $\partial(G,\Pc)$ is equivariantly homeomorphic to $\Lc(G)$. 
\end{enumerate}
\end{theorem}

We also note that if $\Gamma$ is relatively hyperbolic with respect to $\Pc$, then $\Gamma$ acts as a geometrically finite convergence group on the Bowditch boundary.

\section{Finding properly embedded simplices}\label{sec:finding properly embedded simplices}

For convex co-compact groups, if there is a line segment in the ideal boundary of the convex hull, then there is a properly embedded simplex in the convex hull~\cite[Theorem 1.15]{DGF2017}. This fails for naive convex co-compact subgroups, see~\cite[Section 2.3]{IZ2019b}. Instead we will show that if there exists a line segment in the ideal boundary which is sufficiently long (relative to the Hilbert metric of the face containing it), then there exists a properly embedded simplex. 

In the following proposition, $\dist_{\Pb}$ denotes some distance on $\Pb(\Rb^d)$ induced by a Riemannian metric. 

\begin{proposition}\label{prop:finding_simplices} Suppose that $(\Omega, \Cc, \Gamma)$ is a naive convex co-compact triple and $q \in \Cc$. For any $r, \epsilon > 0$ there exists $L=L(q,r,\epsilon) \geq 0$ with the following property: if
\begin{enumerate}
\item $a,b \in \partiali\Cc$ are contained in a boundary face $F$ of $\Omega$ with $\dist_F(a,b) \geq L$, 
\item $m \in (a,b)$ is the midpoint of $[a,b]$ relative to $\dist_F$, and
\item $p \in \Pb(\Span\{ a,b,q\}) \cap \Cc$ is sufficiently close to $m$ (in the metric $\dist_{\Pb}$), 
\end{enumerate}
then there exists a properly embedded simplex $S \subset \Cc$ of dimension at least two with 
\begin{align*}
\Pb(\Span\{ a,b,q\}) \cap \Bc_\Omega(p,r) \subset \Nc_\Omega(S, \epsilon).
\end{align*}
\end{proposition}

\begin{proof} Suppose for a contradiction that the proposition is false for some choice of $r,\epsilon > 0$. Then for every $n \in \Nb$ there exist $a_n, b_n \in \partiali \Cc$ and a sequence $\{p_{n,j}\}_{j \geq 1}$ where 
\begin{enumerate}
\item $a_n,b_n$ are contained in a boundary face $F_n$ and $\dist_{F_n}(a_n,b_n) > n$, 
\item if $m_n \in (a_n,b_n)$ is the midpoint of $[a_n,b_n]$ relative to $\dist_{F_n}$, then $m_n = \lim_{j \rightarrow \infty} p_{n,j}$, 
\item $p_{n,j} \in \Pb(\Span\{ a_n,b_n,q\}) \cap \Cc$ for all $j \geq 1$, and 
\item $\Pb(\Span\{ a_n,b_n,q\}) \cap \Bc_\Omega(p_{n,j},r)$ is not contained in the $\epsilon$-neighborhood of any properly embedded simplex in $\Cc$ of dimension at least two. 
\end{enumerate}

We first claim that for every $n$ there exists $j_n$ such that 
\begin{align}
\label{eqn:dist_to_sides} 
\dist_\Omega( p_{n,j_n}, (a_n, q] \cup [q,b_n) ) \geq n/2.
\end{align}
 Fix $n$ and suppose not. Then for every $j \in \Nb$ we can find $a_{n,j} \in (a_n, q]$ and  $b_{n,j} \in (b_n, q]$ with 
 \begin{align*}
\dist_\Omega\Big(p_{n,j},\{a_{n,j}, b_{n,j}\}\Big) \leq  n/2.
 \end{align*}
 Since $\lim_{j \rightarrow \infty} p_{n,j}=m_n$, we must have $\lim_{j \rightarrow \infty} a_{n,j} = a_n$ and $\lim_{j \rightarrow \infty} b_{n,j} = b_n$. Further, by the definition of the Hilbert metric  
   \begin{align*}
n/2 \geq \limsup_{j \rightarrow \infty} \dist_\Omega\Big(p_{n,j},\{a_{n,j}, b_{n,j}\}\Big) \geq \dist_{F_n}\Big(m_{n},\{a_{n}, b_{n}\}\Big) > n/2.
 \end{align*}
 So we have a contradiction and hence for each $n$ such a $j_n$ exists.

After possibly passing to a subsequence we can find a sequence $\{\gamma_n\}$ in $\Gamma$ such that as $n \to \infty$,
\begin{align*}
\gamma_n(p_{n,j_n}) \rightarrow p_\infty \in \Cc,& \quad \gamma_n(a_n) \rightarrow a_\infty\in \overline{\Cc}, \quad \gamma_n(b_n) \rightarrow b_\infty\in \overline{\Cc}, \\
&  \text{and} \quad \gamma_n(q) \rightarrow q_\infty\in \overline{\Cc}.
\end{align*}
By construction $[a_\infty,b_\infty] \subset \partiali\Cc$ and by Equation~\eqref{eqn:dist_to_sides} we see that
\begin{align*}
[a_\infty, q_\infty] \cup [q_\infty, b_\infty] \subset \partiali \Cc.
\end{align*}
Hence $a_\infty,b_\infty,q_\infty$ are the vertices of a properly embedded two-dimensional simplex $S \subset \Cc$ which contains $p_\infty$. So for $n$ large we have 
\begin{align*}
\Pb(\Span\{ \gamma_n(a_n),\gamma_n(b_n),\gamma_n(q)\}) \cap \Bc_\Omega(\gamma_n(p_{n,j_n}),r) \subset \Nc_\Omega(S, \epsilon),
\end{align*}
which implies that
\begin{align*}
\Pb(\Span\{ a_n,b_n,q\}) \cap \Bc_\Omega(p_{n,j_n},r) \subset \Nc_\Omega(\gamma_n^{-1}S, \epsilon).
\end{align*}
Thus we have a contradiction. 
\end{proof}

\section{Properties of peripheral families}\label{sec:properties of peripheral families}

\begin{theorem}  
\label{thm:boundary_faces}
Suppose that $(\Omega, \Cc, \Gamma)$ is a naive convex co-compact triple and $\Xc$ is a peripheral family of $(\Omega, \Cc, \Gamma)$ (i.e. $\Xc$ is $\Gamma$-invariant, strongly isolated, and coarsely contains all properly embedded simplices in $\Cc$). Then:
\begin{enumerate}
\item $\Gamma$ has finitely many orbits in $\Xc$.
\item If $X \in \Xc$, then $\Stab_{\Gamma}(X)$ acts co-compactly on $X$. In particular, $\Stab_{\Gamma}(X)$ is finitely generated. 
\item If $X_1, X_2 \in \Xc$ are distinct, then $F_\Omega(\partiali X_1) \cap F_\Omega(\partiali X_2) = \emptyset$. 
\item There exists $L > 0$ such that: if $x \in \partiali\Cc$ and $\diam_{F_\Omega(x)} \left( F_\Omega(x) \cap \partiali\Cc \right)\geq L$, then $x \in F_\Omega(\partiali X)$ for some $X \in \Xc$. 
\item There exists $R > 0$ such that: if $X \in \Xc$ and $x \in \partiali X$, then 
\begin{align*}
\dist_{F_\Omega(x)}^{\Haus}\left( F_\Omega(x) \cap \partiali \Cc, F_\Omega(x) \cap \partiali X\right) \leq R.
\end{align*}
\item If $X \in \Xc$, $x \in \partiali X$, and $[x,y] \subset \partiali\Cc$, then $[x,y] \subset F_\Omega(\partiali X)$. 
\item If $[x,y] \subset \partiali \Cc$, then either $F_{\Omega}(x)=F_{\Omega}(y)$, or there exists $X \in \Xc$ such that $x,y \in F_{\Omega}(\partiali X)$. 
\item If $X \in \Xc$, then $\Stab_{\Gamma}(X)$ acts co-compactly on 
$$
\CH_\Omega(F_\Omega(\partiali X) \cap \partiali \Cc) \cap \Omega.
$$
\end{enumerate}
\end{theorem}

For the rest of the section fix $\Omega$, $\Cc$, $\Gamma$, and $\Xc$ satisfying the hypothesis of the theorem. 

Since $\Xc$ coarsely contains all properly embedded simplices in $\Cc$, there exists $D_2 > 0$ such that: if $S \subset \Cc$ is a properly embedded simplex of dimension at least two, then there exists $X \in \Xc$ with 
\begin{align}
\label{eqn:D2-for-thm-bd-face}
S \subset \Nc_\Omega(X,D_2). 
\end{align}

\begin{lemma} \label{lem:finitely_many_simplices_in_compact_set}
If $K \subset \Cc$ is compact, then the set $\{ X \in \Xc : X \cap K \neq \emptyset\}$ is finite. \end{lemma}

\begin{proof} Suppose not. Then there exist an infinite sequence of pairwise distinct elements $\{X_n\}$ in $\Xc$ where $X_n \cap K \neq \emptyset$ for all $n$. As each $X_n$ is unbounded,  we can fix $k_n \in X_n \cap K$ and $x_n \in \partiali X_n$. Passing to a subsequence we can suppose that $k_n \rightarrow k \in K$ and $x_n \rightarrow x \in \partiali\Cc$. Then
\begin{align*}
\liminf_{n,m\rightarrow \infty} & \diam_\Omega( \Nc_\Omega( X_n, 1) \cap  \Nc_\Omega( X_m, 1)) \\
& \geq \liminf_{n,m\rightarrow \infty} \diam_\Omega( \Nc_\Omega( [k_n,x_n), 1) \cap  \Nc_\Omega( [k_m,x_m), 1))= \infty
\end{align*}
and so by the strongly isolation property $X_n = X_m$ when $n,m$ are sufficiently large. So we have a contradiction. 
\end{proof}

\begin{lemma}[Part (1)]
$\Gamma$ has finitely many orbits in $\Xc$.
\end{lemma}

\begin{proof}
Since $\Gamma$ acts co-compactly on $\Cc$, there exists a compact set $K\subset \Cc$ such that $\Gamma \cdot K =\Cc$. Then for each $X \in \Xc$ there exists $g \in \Gamma$ such that $K \cap gX  \neq \emptyset$. So by Lemma \ref{lem:finitely_many_simplices_in_compact_set}, there exist $X_1, \ldots, X_m$ such that $\Xc=\sqcup_{i=1}^m \Gamma \cdot X_i.$
\end{proof}

\begin{lemma}[Part (2)] If $X \in \Xc$, then $\Stab_{\Gamma}(X)$ acts co-compactly on $X$. 

\end{lemma}

\begin{proof}
This argument is standard, see~\cite[Proposition 4.0.4]{W1996},~\cite[Theorem 3.7]{H2005}, or ~\cite[Section 3.1]{HK2005}. 

Fix a compact set $K \subset \Cc$ such that $\Gamma \cdot K = \Cc$. Let 
$$
\mathcal{G} = \{ g \in \Gamma : X \cap gK  \neq \emptyset\}.
$$
Then $X \subset \mathcal{G} \cdot K$. Also, by the previous lemma 
$$
\{ g^{-1} X : g \in \mathcal{G} \} = \{ h_1^{-1}X,\dots, h_m^{-1}X\}
$$
for some $h_1,\dots,h_m \in \mathcal{G}$. Notice that $g^{-1}X = h_j^{-1} X$ if and only if $gh_j^{-1} \in \Stab_\Gamma(X)$,  hence 
$$
\Gc = \cup_{j=1}^m  \Stab_\Gamma(X)h_j.
$$ 
Finally, if $\hat{K} := \cup_{j=1}^m h_j K$, then $\hat{K}$ is compact and $X \subset \Stab_\Gamma(X) \cdot \hat{K}$. So $X = \Stab_\Gamma(X) \cdot (X \cap \hat{K})$. As $X$ is closed, $X \cap \hat{K}$ is compact and thus $\Stab_{\Gamma}(X)$ acts co-compactly on $X$.

Now, since $X$ is convex, the metric space $(X, \dist_\Omega)$ is geodesic. So by the fundamental lemma of geometric group theory $\Stab_{\Gamma}(X)$ is finitely generated. 
\end{proof}

\begin{lemma}[Part (3)]\label{lem:intersecting_faces} If $X_1, X_2 \in \Xc$ are distinct, then $F_\Omega(\partiali X_1) \cap F_\Omega(\partiali X_2) = \emptyset$.  \end{lemma}

\begin{proof} Suppose $X_1, X_2 \in \Xc$ and $x \in F_\Omega(\partiali X_1) \cap F_\Omega(\partiali X_2)$. Then there exists $x_1 \in \partiali X_1$ and $x_2 \in \partiali X_2$ with 
\begin{align*}
x_1, x_2 \in F_\Omega(x).
\end{align*}
Fix $q_1 \in X_1$ and $q_2 \in X_2$. Then Proposition~\ref{prop:Hausdorff_distance_between_lines} implies that
\begin{align*}
\dist_\Omega^{\Haus}\Big( [q_1, x_1), [q_2, x_2) \Big) \leq \max\left\{\dist_\Omega(q_1,q_2), \dist_{F_\Omega(x)}(x_1, x_2)\right\}.
\end{align*}
So for $r > \max\left\{\dist_\Omega(q_1,q_2), \dist_{F_\Omega(x)}(x_1, x_2)\right\}$, we have 
\begin{align*}
\diam_\Omega \left( \Nc_\Omega(X_1, r) \cap \Nc_\Omega(X_2, r) \right) = \infty.
\end{align*}
Since $\Xc$ is strongly isolated, then $X_1 = X_2$. 
\end{proof}

\begin{lemma}\label{lem:segments_in_bd} There exists $L > 0$ such that: if
\begin{enumerate}
\item $a,b \in \partiali \Cc$ are in the same boundary face $F$ of $\Omega$, 
\item $\dist_F(a,b) \geq L$, and
\item $m$ is the midpoint of $[a,b]$ with respect to $\dist_F$,
\end{enumerate}
then there exists $X \in \Xc$ and $x \in \partiali X \cap F$ such that 
\begin{align*}
\dist_{F}(m,x) \leq D_2+1. 
\end{align*}
\end{lemma}

\begin{proof} Since $\Xc$ is strongly isolated, there exists $r> 0$ such that
\begin{align*}
\diam_\Omega \left( \Nc_\Omega(X_1, D_2+1) \cap \Nc_\Omega(X_2, D_2+1) \right) < r
\end{align*}
for all distinct pairs $X_1, X_2 \in \Xc$ where $D_2$ is the constant in Equation \eqref{eqn:D2-for-thm-bd-face}. Fix $q \in \Cc$ and let $L > 0$ satisfy Proposition~\ref{prop:finding_simplices} for $q \in \Cc$ and constants $\epsilon=1$ and $r$ as above. Now suppose 
\begin{enumerate}
\item $a,b \in \partiali \Cc$ are in the same boundary face $F$ of $\Omega$, 
\item $\dist_F(a,b) \geq L$, and
\item $m$ is the midpoint of $[a,b]$ with respect to $\dist_F$.
\end{enumerate}
By our choice of $L$, there exists $q^\prime \in (m,q]$ such that: if $p \in (m,q^\prime]$ then there exists some properly embedded simplex $S_p \subset \Cc$ of dimension at least two such that 
\begin{align*}
\Pb(\Span\{ a,b,q\}) \cap \Bc_\Omega(p,r) \subset \Nc_\Omega(S_p, 1).
\end{align*}
Then, by our choice of $D_2$ in Equation \eqref{eqn:D2-for-thm-bd-face}, there exists $X_p \in \Xc$ such that 
\begin{align*}
\Pb(\Span\{ a,b,q\}) \cap \Bc_\Omega(p,r) \subset \Nc_\Omega(X_p, D_2+1).
\end{align*}

We claim that $X_p$ does not depend on $p \in (m,q^\prime]$. To verify this it is enough to fix $p_1,p_2 \in (m,q^\prime]$ with $\dist_\Omega(p_1,p_2) < r$ and show that $X_{p_1}$ and $X_{p_2}$ coincide. For such $p_1, p_2 \in (m,q^\prime]$ we have 
\begin{align*}
 \Pb(\Span\{ a,b,q\}) \cap \Bc_\Omega(p_1, r) \cap \Bc_\Omega(p_2, r)   \subset  \Nc_\Omega(X_{p_1}, D_2+1) \cap  \Nc_\Omega(X_{p_2}, D_2+1)
 \end{align*}
and 
\begin{align*}
\diam_\Omega \Big(  \Pb(\Span\{ a,b,q\}) \cap  \Bc_\Omega(p_1, r) \cap \Bc_\Omega(p_2, r)  \Big) \geq r.
\end{align*}
So $X_{p_1}=X_{p_2}$ by our choice of $r$.

Now let $X:=X_p$ for $p \in (m,q^\prime]$. Then
\begin{align*}
(m,q^\prime] \subset \Nc_\Omega(X, D_2+1)
\end{align*}
which implies, by Proposition~\ref{prop:asymp_sequences}, that there exists $x \in \partiali X \cap F$ such that 
\begin{equation*}
\dist_{F}(m,x) \leq D_2+1. \qedhere
\end{equation*}
\end{proof}

\begin{lemma}[Part (4)]\label{lem:part4}
If $x \in \partiali\Cc$ and 
$$
\diam_{F_\Omega(x)} \left( F_\Omega(x) \cap \partiali\Cc \right)\geq L,
$$ 
then $x \in F_\Omega(\partiali X)$ for some $X \in \Xc$. 
\end{lemma}

\begin{proof} This follows immediately from Lemma~\ref{lem:segments_in_bd}. \end{proof}

\begin{lemma}[Part (5)]
\label{lem:boundary-faces-of-X}
If $X \in \Xc$ and $x \in \partiali X$, then 
\begin{align*}
\dist_{F_\Omega(x)}^{\Haus}\left( F_\Omega(x) \cap \partiali \Cc, F_\Omega(x) \cap \partiali X\right) \leq \max\{ L, 2D_2+2\}. 
\end{align*}
\end{lemma}

\begin{proof} Fix $X \in \Xc$ and $x \in \partiali X$. Since $F_\Omega(x) \cap \partiali X \subset F_\Omega(x) \cap \partiali \Cc$, it suffices to show that 
$$
\sup_{y \in F_\Omega(x) \cap \partiali \Cc} \dist_{F_\Omega(x)}(y, F_\Omega(x) \cap \partiali X ) \leq \max\{ L, 2D_2+2\}.
$$
Fix $y \in F_\Omega(x) \cap \partiali \Cc$ and suppose for a contradiction that 
\begin{align*}
\dist_{F_\Omega(x)}(y, F_\Omega(x) \cap \partiali X ) > \max\{ L, 2D_2+2\}.
\end{align*}
By changing $x$ we may suppose that 
\begin{align*}
\dist_{F_\Omega(x)}(y, F_\Omega(x) \cap \partiali X ) = \dist_{F_\Omega(x)}(y,x). 
\end{align*}
Let $m$ be the midpoint of $[x,y]$ with respect to $\dist_{F_\Omega(x)}$. Then by Lemma~\ref{lem:segments_in_bd} there exist $X^\prime \in \Xc$ and $x^\prime \in \partiali X^\prime$ such that 
\begin{align*}
\dist_{F_\Omega(x)}(m,x^\prime) \leq D_2+1. 
\end{align*}
However, then $x \in F_\Omega(\partiali X) \cap F_\Omega(\partiali X^\prime) \neq \emptyset$ and so by Lemma~\ref{lem:intersecting_faces} we have $X=X^\prime$. So 
\begin{align*}
D_2+1 &< \frac{1}{2}\dist_{F_\Omega(x)}(y, F_\Omega(x) \cap \partiali X ) =  \dist_{F_\Omega(x)}(m,  F_\Omega(x) \cap \partiali X ) \\
&\leq  \dist_{F_\Omega(x)}(m,x^\prime) \leq D_2+1
\end{align*}
and we have a contradiction. 
\end{proof}

\begin{lemma}[Part (7)]
\label{lem:line-in-boundary-dichotomy}
If $[x,y] \subset \partiali \Cc$, then either $F_{\Omega}(x)=F_{\Omega}(y)$, or there exists $X \in \Xc$ such that $x,y \in F_{\Omega}(\partiali X)$. 
\end{lemma}

\begin{proof} We suppose that $F_\Omega(x) \neq F_\Omega(y)$ and show that $x,y \in F_{\Omega}(\partiali X)$ for some $X \in \Xc$. Fix $z \in (x,y)$. Notice that since $F_{\Omega}(x) \neq F_{\Omega}(y)$, 
$$
\diam_{F_\Omega(z)}\Big( (x,y) \Big) =\infty.
$$
So by Lemma~\ref{lem:part4} there exists $X \in \Xc$ with $z \in F_\Omega( \partiali X)$. Then Lemma \ref{lem:boundary-faces-of-X} implies that
$$
(x,y) \subset \Nc_{F_\Omega(z)}\left( F_\Omega(z) \cap \partiali X,  \max\{ L, 2D_2+2\}\right).
$$
Finally Proposition~\ref{prop:asymp_sequences} implies that $x,y \in F_\Omega(\partiali X)$.  
\end{proof}

\begin{lemma}[Part (6)]\label{lem:part 6} If $X \in \Xc$, $x \in \partiali X$, and $[x,y] \subset \partiali\Cc$, then $[x,y] \subset F_\Omega(\partiali X)$. 
\end{lemma}

\begin{proof} This follows immediately from Lemmas~\ref{lem:line-in-boundary-dichotomy} and~\ref{lem:intersecting_faces}. \end{proof}

\begin{lemma}[Part (8)]
If $X \in \Xc$, then $F_\Omega(\partiali X) \cap \partiali \Cc$ is closed and $\Stab_{\Gamma}(X)$ acts co-compactly on $\CH_\Omega(F_\Omega(\partiali X) \cap \partiali \Cc) \cap \Omega$. 
\end{lemma}

\begin{proof} We first verify that $F_\Omega(\partiali X) \cap \partiali \Cc$ is closed. Suppose $\{x_n\}$ is a sequence in $F_\Omega(\partiali X) \cap \partiali \Cc$ converging to some $x \in \partiali\Cc$ (note $\partiali \Cc$ is closed). Then for each $n$ there exists $x_n^\prime \in \partiali X$ with $x_n \in F_\Omega(x_n^\prime)$. In particular, $[x_n, x_n^\prime] \subset \partiali \Cc$. Passing to a subsequence we can suppose that $x_n^\prime \rightarrow x^\prime$. Then $x^\prime \in \partiali X$ and $[x,x^\prime] \subset \partiali \Cc$. Then  Lemma \ref{lem:part 6} implies that $x \in F_{\Omega}(\partiali X)$. Hence  $F_\Omega(\partiali X) \cap \partiali \Cc$ is closed.

Now $\Cc_X : = \CH_\Omega(F_\Omega(\partiali X) \cap \partiali\Cc) \cap \Omega$ is closed in $\Cc$ and $\Stab_{\Gamma}(X)$ acts co-compactly on $X$, so it suffices to show that $\Cc_X$ is contained in a bounded neighborhood of $X$. Fix $p \in \Cc_X$. Then there exist $p_1, \dots, p_m \in F_\Omega(\partiali X) \cap \partiali \Cc$ such that $p \in \CH_\Omega(p_1,\dots, p_m)$. Then by Lemma~\ref{lem:boundary-faces-of-X} there exist $p_1^\prime, \dots, p_m^\prime \in \partiali X$ such that 
$$
\dist_{F_\Omega(p_j)}(p_j, p_j^\prime) \leq \max\{ L, 2D_2+2\} \quad \text{for } j=1,\dots, m.
$$
Then by Proposition~\ref{prop:Hausdorff_distance_between_ch}
$$
\dist_\Omega(p, X) \leq \max\{ L, 2D_2+2\}.
$$
Since $p \in \Cc_X$ was arbitrary, $\Cc_X \subset \Nc_\Omega(X, \max\{ L, 2D_2+2\})$. 
\end{proof}

\section{Naive convex co-compactness of peripheral subgroups}
\label{sec:peripheral subgroups are ncc}

\begin{proposition}
\label{prop:rh_implies_isolation}
Suppose that $(\Omega, \Cc, \Gamma)$ is a naive convex co-compact triple and $\Gamma$ is relatively hyperbolic with respect to $\{P_1,\dots, P_m\}$. For $1 \leq j \leq m$, let 
$$
X_j:=\CH_\Omega (\limset(P_j) \cap \partiali \Cc) \cap \Omega. 
$$ 
Then:
\begin{enumerate}
\item Each $(\Omega,X_j,P_j)$ is a naive convex co-compact triple.
\item $(\Cc, \dist_\Omega)$ is relatively hyperbolic with respect to $\Xc:=\Gamma \cdot \{X_1, \dots, X_m\}$. 
\end{enumerate}
\end{proposition}

The rest of the section is devoted to the proof of the proposition. Suppose $(\Omega, \Cc, \Gamma)$ is a naive convex co-compact triple and $\Gamma$ is relatively hyperbolic with respect to $\{P_1,\dots, P_m\}$.

Fix a point $p_0 \in \Cc$. The fundamental lemma of geometric group theory implies that $(\Cc, \dist_\Omega)$ is relatively hyperbolic with respect to $\Gamma \cdot \{ P_1(p_0), \dots, P_m(p_0)\}$.

For each $1 \leq j \leq m$, let $\Lc_j \subset \partiali\Cc$ denote the accumulation points of the orbit $P_j(p_0)$ and let 
$$
\hat{X}_j ={\rm ConvHull}_\Omega(  \Lc_j) \cap \Omega.
$$

\begin{lemma} $\hat{X_j}$ is  non-empty and $P_j$ acts co-compactly on $\hat{X}_j$. 
\end{lemma}

\begin{proof} For $n \geq 1$ let $Y_j^{(n)}$ be the points which are contained in the convex hull of at most $n$ points in $P_j(p_0)$.  By Carath\'eodory's convex hull theorem, $Y_j^{(n)} = Y_j^{(d)}$ for all $n \geq d$. Further, the closure of $Y_j^{(d)}$ in $\Omega$ contains $\hat{X}_j$. 

By Theorem \ref{thm:rh_ds} part \ref{thm:quasi-convexity} there exists $t > 0$ such that if $r \geq 1$ and $p,q \in \Cc \cap \Nc_\Omega( P_j(p_0), r)$, then $[p,q] \subset \Nc_\Omega(P_j(p_0), tr)$. 

We claim that 
\begin{equation}
\label{eqn:inclusion in induction}
Y_j^{(n)} \subset \Nc_\Omega( P_j(p_0), (n-1)t)
\end{equation}
for all $n \geq 2$. The base case follows by our choice of $t$. Suppose Equation~\eqref{eqn:inclusion in induction} holds for some $n$. If $p \in Y_j^{(n+1)}$, then there exist $p_1 \in P_j(p_0)$ and $p_2 \in Y_j^{(n)}$ such that $p \in [p_1, p_2]$. Then, by induction there exists $q \in P_j(p_0)$ with $\dist_\Omega(p_2, q) < (n-1)t$. Then 
$$
[p_1, q] \subset Y_j^{(2)} \subset \Nc_\Omega( P_j(p_0), t)
$$
and by Proposition~\ref{prop:Hausdorff_distance_between_lines}
$$
\dist_\Omega( p, [p_1,q]) \leq \dist_\Omega^{\Haus}([p_1,p_2],[p_1,q]) \leq \dist_\Omega(p_2,q) < (n-1)t.
$$ 
So $p \in \Nc_\Omega( P_j(p_0), nt)$. Since  $p \in Y_j^{(n+1)}$ was arbitrary, this completes the induction step. Thus Equation~\eqref{eqn:inclusion in induction} holds for all $n \geq 2$. 

Then 
$$
\hat{X}_j \subset \overline{Y_j^{(d)}} \subset \overline{\Nc_\Omega( P_j(p_0), (d-1)t)}
$$ 
and so
\begin{align*}
\hat{X_j} = P_j\cdot \left( \overline{\Bc_\Omega(p_0, (d-1)t)} \cap \hat{X}_j\right).
\end{align*}
Thus $P_j$ acts co-compactly on $\hat{X}_j$.

It remains to show that $\hat{X_j} \neq \emptyset$. Since $P_j$ is infinite and acts properly on $\Omega$, there exists $\{g_n\}$ in $P_j$ such that 
$$
\lim_{n \rightarrow \infty} \dist_\Omega( p_0, g_n(p_0)) = \infty.
$$
Let $m_n$ be the midpoint of $\left[p_0,g_n(p_0)\right]$ with respect to $\dist_\Omega$. By Equation \eqref{eqn:inclusion in induction},  
$$m_n \in Y^{(2)}_j \subset \Nc_{\Omega}(P_j (p_0) , t).$$
Thus there exists $\{h_n\}$ in $P_j$ such that $h_n(m_n) \in \Bc_{\Omega}(p_0,t)$. Then, up to passing to a subsequence, we can assume that $h_n(m_n) \to m \in \Omega$ and $h_n[p_0,g_n(p_0)] \to [x,y]$.  Then $x,y \in \Lc_j$ and so 
\begin{equation*}
m \in (x,y) \subset \hat{X_j}. \qedhere
\end{equation*}
\end{proof}

The fundamental lemma of geometric group theory implies that $(\Cc, \dist_\Omega)$ is relatively hyperbolic with respect to $\Gamma \cdot \{\hat{X}_1, \dots, \hat{X}_m\}$. In particular, by Proposition~\ref{prop:immediate_consequences_of_DS05} and Theorem~\ref{thm:boundary_faces} part (5) there exists $R > 0$ such that: if $1 \leq j \leq m$ and $x \in \partiali \hat{X}_j$, then 
\begin{equation}
\label{eqn:size_of_boundary_pieces}
\dist_{F_\Omega(x)}^{\Haus}\Big(   F_\Omega(x) \cap \partiali \Cc, F_\Omega(x) \cap \partiali \hat{X}_j  \Big) \leq R. 
\end{equation}

\begin{lemma} If $1 \leq j \leq m$, then 
\begin{equation*}
\dist_{F_\Omega(x)}^{\Haus}\Big( \hat{X}_j, X_j \Big) \leq R. 
\end{equation*}
Hence $P_j$ acts co-compactly on $X_j$ and $(\Cc, \dist_\Omega)$ is relatively hyperbolic with respect to $\Gamma \cdot \{X_1, \dots, X_m\}$.
\end{lemma}

\begin{proof} By definition $\hat{X}_j \subset X_j$ and by the previous lemma, $P_j$ acts co-compactly on $\hat{X_j}$. Thus it suffices to show that 
$$
X_j \subset \overline{\Nc_\Omega(\hat{X}_j, R)}. 
$$
Fix $q \in X_j$. Then there exist $q_1,\dots, q_m \in \Lc_\Omega(P_j) \cap \partiali \Cc$ where 
$$
q \in {\rm ConvHull}_\Omega(q_1,\dots,q_m).
$$
For each $1 \leq i \leq m$, there exist $p_i \in \Omega$ and a sequence $\{\gamma_{i,n}\}_{n \geq 1}$ in $P_j$ with $\gamma_{i,n}(p_i) \rightarrow q_i$. Passing to subsequences we can suppose that 
$$
\gamma_{i,n}(p_0) \rightarrow q_i^\prime \in \Lc_j \subset \partiali \hat{X}_j
$$
 for all $1 \leq i \leq m$. By Proposition~\ref{prop:asymp_sequences}, $F_\Omega(q_i) = F_\Omega(q_i^\prime)$. So by Equation~\eqref{eqn:size_of_boundary_pieces} there exists $q_i^{\prime\prime} \in \partiali \hat{X}_j$ such that 
$\dist_{F_\Omega(q_i)}(q_i, q_i^{\prime\prime}) \leq R.$
Then by Proposition~\ref{prop:Hausdorff_distance_between_ch} 
\begin{align*}
\dist_\Omega(q, \hat{X}_j) & \leq \dist_\Omega^{\Haus}\Big({\rm ConvHull}_\Omega(q_1,\dots,q_m) \cap \Omega, {\rm ConvHull}_\Omega(q_1^{\prime\prime},\dots,q_m^{\prime\prime}) \cap \Omega \Big)\\
&  \leq \max_{1 \leq i \leq m} \dist_{F_\Omega(q_i)}(q_i, q_i^{\prime\prime}) \leq R.  \qedhere
\end{align*}
\end{proof} 

This finishes the proof of Proposition \ref{prop:rh_implies_isolation}.

\section{Co-compact boundary actions}\label{sec: co-compact boundary actions}

Given a naive convex co-compact triple $(\Omega, X, G)$, it seems natural to ask if $G$ acts co-compactly on $\partial \Omega - F_\Omega(\partiali X)$ or more generally on $\partiali\Cc - F_\Omega(\partiali X)$ when $\Cc \subset \Omega$ is a closed $G$-invariant convex subset containing $X$. The next theorem gives a sufficient condition for this to occur. 

\begin{theorem}\label{thm:co_compact} Suppose that $(\Omega, X,G)$ is a naive convex co-compact triple and $\Cc \subset \Omega$ is a closed $G$-invariant convex subset containing $X$ with the property: if $x \in \partiali X$ and $[x,y] \subset \partiali\Cc$, then $[x,y] \subset F_\Omega(\partiali X)$. Then $G$ acts co-compactly on $\partiali \Cc - F_\Omega(\partiali X)$.
\end{theorem}

Letting $\Cc = \Omega$ and changing notation, we obtain: 

\begin{corollary} Suppose that $(\Omega, \Cc, \Gamma)$ is a naive convex co-compact triple with the property: if $x \in \partiali \Cc$ and $[x,y] \subset \partial\Omega$, then $[x,y] \subset F_\Omega(\partiali \Cc)$. Then $\Gamma$ acts co-compactly on $\partial \Omega - F_\Omega(\partiali \Cc)$.
\end{corollary}

\begin{remark}
\label{rem:boundary-action-and-parabolic-point}
Notice that if $(\Omega,\Cc,\Gamma)$ is a naive convex co-compact triple, $\Xc$ is a peripheral family of $(\Omega,\Cc,\Gamma)$, and $X \in \Xc$, then $(\Omega, X, \Stab_{\Gamma}(X))$ and $\Cc$ satisfy the conditions in Theorem~\ref{thm:co_compact}  (see Theorem \ref{thm:boundary_faces} parts (6) and (8)). We will use this in the proof of (2) $\Rightarrow$ (1) in Theorem \ref{thm:main} (see Section \ref{sec:proof_main} below) to show that non-conical points are bounded parabolic points. \end{remark}

For the rest of the section fix $\Omega$, $\Cc$, $X$, $G$ satisfying the hypothesis of the theorem. The key idea will be to study the following  ``closest point projection.'' 

\begin{definition} Define
\begin{align*}
\pi_X : \overline{\Cc} \rightarrow \{\text{subsets of $X$}\}
\end{align*}
as follows: 
\begin{itemize}
\item For $p \in \Cc$ let 
\begin{align*}
\pi_X(p) := \{ x \in X : \dist_\Omega(p,x) = \dist_\Omega(p,X)\}
\end{align*}
denote the points in $X$ closest to $p$. 
\item For $y \in \partiali\Cc$ let $\pi_X(y) \subset X$ denote the set of all points $x \in X$ where there exist sequences $\{p_n\}$ and $\{x_n\}$ such that $p_n \in \Cc$, $x_n \in \pi_X(p_n)$,  $p_n \rightarrow y$, and $x_n \rightarrow x$. 
\end{itemize}
\end{definition}

We start by observing the following. 

\begin{observation}\label{obs:closest_pt_proj} \ 
\begin{enumerate}
\item If $p \in \Cc$, then $\pi_X(p)$ is a compact convex subset of $X$. 
\item If $p \in \Cc$ and $x \in \pi_X(p)$, then $x \in \pi_X(q)$ for all $q \in [x,p]$.
\item If $g \in G$, then $\pi_X \circ g = g \circ \pi_X$. 
\end{enumerate}
\end{observation}

\begin{proof} Notice that if $p \in \Cc$, then 
$$
\pi_X(p) = X \cap \{ x \in \Omega : \dist_\Omega(p,x) \leq \dist_\Omega(p,X)\}.
$$
Part (1) is then a consequence of the fact that closed metric balls in the Hilbert metric are convex and compact (see Proposition~\ref{prop:Hausdorff_distance_between_lines} above). Part (2) follows from the fact that $[x,p]$ can be parametrized to be a geodesic in the Hilbert metric and part (3) is by definition. 
\end{proof} 

\begin{lemma}\label{lem:non-empty-projection} If $y \in \partiali\Cc - F_\Omega(\partiali X)$, then $\pi_X(y) \subset X$ is a non-empty bounded set. \end{lemma}

\begin{proof} It suffices to fix sequences $\{p_n\}$ and $\{x_n\}$ such that $p_n \in \Cc$, $x_n \in \pi_X(p_n)$, $p_n \rightarrow y$, and $x_n \rightarrow x \in \overline{X}$, then prove that $x \in X$. 

Suppose for a contradiction that $x \in \partiali X$. Since $y \in \partiali\Cc - F_\Omega(\partiali X)$ our hypothesis on $X$ says that $(x,y) \subset \Omega$. Fix $v \in (x,y)$. Then pick a sequence $v_n \in (x_n,p_n)$ with $v_n \rightarrow v$. Then $x_n \in \pi_X(v_n)$ by Observation~\ref{obs:closest_pt_proj} part (2). However then
\begin{align*}
\dist_\Omega(v,X) = \lim_{n \rightarrow \infty} \dist_\Omega(v_n,X) =  \lim_{n \rightarrow \infty} \dist_\Omega(v_n,x_n) = \lim_{n \rightarrow \infty} \dist_\Omega(v,x_n) =\infty
\end{align*}
and we have a contradiction.  
\end{proof}

\begin{lemma}\label{lem:empty_closest_pt} If $y \in \partiali\Cc \cap F_\Omega(\partiali X)$, then $\pi_X(y) = \emptyset$. \end{lemma}

\begin{proof} Suppose for a contradiction that there exist sequences $\{p_n\}$ and $\{x_n\}$ where $p_n \in \Cc$, $x_n \in \pi_{X}(p_n)$, $p_n \rightarrow y$, and $x_n \rightarrow x \in X$. 

By hypothesis, there exists $w \in \partiali X$ with $y \in F_\Omega(w)$. Fix $q \in (x,y)$ such that $\dist_\Omega(q,x) > \dist_{F_\Omega(w)}(y,w)$. Then fix $q_n \in (x_n, p_n)$ such that $q_n \rightarrow q$. By Observation~\ref{obs:closest_pt_proj} part (2)
\begin{align*}
\dist_\Omega(q_n, X) = \dist_\Omega(q_n,x_n)
\end{align*}
and so
\begin{align*}
\dist_\Omega(q, X) =\lim_{n \rightarrow \infty} \dist_\Omega(q_n, X)= \dist_\Omega(q, x) > \dist_{F_\Omega(w)}(y,w).
\end{align*}
However this is impossible since Proposition~\ref{prop:Hausdorff_distance_between_lines} implies that
\begin{align*}
\dist_\Omega(q, X) \leq \dist_\Omega \left( q, (x,w) \right) \leq \dist_\Omega^{\Haus} \left( (x,y), (x,w) \right) \leq \dist_{F_\Omega(w)}(y,w).
\end{align*}
So we have a contradiction. 
\end{proof}

\begin{lemma}\label{lem:compact_preimage} If $K \subset X$ is compact, then the set 
\begin{align*}
\wh{K}:=\left\{ y \in \partiali\Cc: \pi_X(y) \cap K \neq \emptyset\right\}
\end{align*}
is compact. 
\end{lemma}

\begin{proof} Consider a sequence $\{y_n\}$ in $\wh{K}$ where $y_n \rightarrow y \in \partiali\Cc$. Fix $x_n \in \pi_X(y_n) \cap K$. Passing to a subsequence we can suppose that $x_n \rightarrow x \in K$. We can also find sequences $\{p_{n,m}\}$ and $\{x_{n,m}\}$ such that $p_{n,m} \in \Cc$, $x_{n,m} \in \pi_X(p_{n,m})$, 
\begin{align*}
y_n = \lim_{m \rightarrow \infty} p_{n,m} \quad \text{and} \quad x_n = \lim_{m \rightarrow \infty} x_{n,m}.
\end{align*}
Then we can pick subsequences $\{p_{n_j, m_j}\}$ and $\{x_{n_j, m_j}\}$ such that 
\begin{align*}
y = \lim_{j \rightarrow \infty} p_{n_j, m_j}\quad \text{and} \quad x = \lim_{j \rightarrow \infty} x_{n_j, m_j}.
\end{align*}
So $x \in \pi_X(y) \cap K$ and hence $y \in \wh{K}$. 
\end{proof}

\begin{proof}[Proof of Theorem~\ref{thm:co_compact}] Fix a compact set $K \subset X$ such that $G \cdot K = X$. Then Lemmas~\ref{lem:compact_preimage} and~\ref{lem:empty_closest_pt} imply that
\begin{align*}
\wh{K} := \{ y \in \partiali\Cc : \pi_X(y) \cap K \neq \emptyset\}
\end{align*}
is compact and contained in $\partiali\Cc - F_\Omega(\partiali X)$. 

We claim that 
\begin{align*}
G  \cdot \wh{K} = \partiali\Cc - F_\Omega(\partiali X).
\end{align*} 
Fix $y \in \partiali\Cc - F_\Omega(\partiali X)$. Then by Lemma~\ref{lem:non-empty-projection} there exists $x \in \pi_X(y)$. Then there exists $g \in G$ such that $g(x) \in K$. Then $g(x) \in  \pi_X(g(y))$ and so $g(y) \in \wh{K}$. So $y \in G \cdot \wh{K}$. Since $y$ was arbitrary, this proves the claim and the theorem. 
\end{proof}

\section{Basic properties of boundary quotients}
\label{sec:bdry_quotient}

For the rest of the section fix a properly convex domain $\Omega \subset \Pb(\Rb^d)$, a closed convex subset $\Cc \subset \Omega$, and a discrete subgroup $\Gamma \leq \Aut(\Omega)$ which preserves $\Cc$. Notice that we do not assume that $\Gamma$ acts co-compactly on $\Cc$. 

Also fix a $\Gamma$-invariant equivalence relation $\sim$ on $\partiali\Cc$ such that the set 
$$
R = \left\{ (x,y)  \in \partiali\Cc \times \partiali\Cc : x \sim y\right\}
$$ 
is closed and $\#(\eqv{\partiali\Cc}) \geq 3$. For each $x \in \partiali\Cc$, let $[x]$ denote the equivalence class of $x$ and let
$$
\Cc_x := \CH_\Omega([x]) \cap\Omega.
$$
Notice that it is possible for $\Cc_x$ to be empty. 

We consider two conditions:
\begin{enumerate}
\item We say that \emph{Condition (1) holds}, if whenever $x,y \in \partiali\Cc$ and $[x,y] \subset \partial \Omega$, then $x \sim y$. 
\item We say that \emph{Condition (2) holds} if there exist $r, D > 0$ such that: if $x \not\sim y$, then 
$$
\diam_\Omega\Big( \Nc_\Omega\left(\Cc_x,r\right) \cap \Nc_\Omega\left(\Cc_y,r\right) \Big) < D.
$$
\end{enumerate} 

We will prove the following results about the quotient space $\eqv{\partiali\Cc}$. These arguments are similar to analogous arguments of Choi~\cite{Choi_book} and Weisman~\cite{W2020}.

\begin{proposition}\label{prop:metrizable} The quotient $\eqv{\partiali\Cc}$ is a compact Hausdorff metrizable space. \end{proposition}

\begin{proposition}\label{prop:convergence_grp} If Condition (1) holds, then $\Gamma$ acts as a convergence group on $\eqv{\partiali\Cc}$.
\end{proposition}

To state the final result, we need a definition. 

\begin{definition} A point $x \in \partial \Omega$ is a \emph{uniformly conical limit point of $\Gamma$ acting on $\Omega$} if for any $p_0 \in \Omega$ the image of $[p_0,x)$ in $\Gamma \backslash \Omega$ is relatively compact. 
\end{definition} 

\begin{proposition}\label{prop:conical_limit_pts} If Conditions (1) and (2) hold, $x \in \partiali\Cc$ is a uniformly conical limit point of $\Gamma$ acting on $\Omega$, and $[x] \subset F_\Omega(x)$, then $[x] \in \eqv{\partiali\Cc}$ is a conical limit point of $\Gamma$ (in the convergence group sense, see Section \ref{sec:yaman_creiteria}). 
\end{proposition} 

\subsection{Proof of Proposition~\ref{prop:metrizable}} 
 The proof is an exercise in point set topology. We include it for the convenience of the reader. We also note that similar arguments appear in~\cite[Proposition 10.5]{Choi_book} and~\cite[Proposition 8.10]{W2020}.  

Since $\partiali\Cc$ is compact and $R$ is closed, $\eqv{\partiali\Cc}$ is compact and Hausdorff. Then by the Uryshon metrization theorem it is enough to show that $\eqv{\partiali\Cc}$ has a countable basis. 

Since $\partiali\Cc$ is a compact metrizable space, we can choose a countable basis $\Uc:=\{U_n:n \in \Nb\}$ for $\partiali\Cc$. Without loss of generality, we can assume that $\Uc$ is closed under finite union. For $n \geq 1$, let 
$$
U_n^*:=\{ x \in \partiali\Cc :[x] \subset U_n \}.
$$
Then $\Uc^*:=\{ U_n^* : U_n^* \neq \emptyset, n \in \Nb\}$ is an open cover of $\partiali\Cc$ (since $\Uc$ is closed under finite union and the quotient map is proper). 

Next let $\pi : \partiali\Cc \rightarrow \eqv{\partiali\Cc}$ denote the quotient map.  If $n \geq 1$, define 
$$
V_n:=\pi(U_n^*).
$$ 
It is clear that $\Vc:=\{V_n: U_n^* \neq \emptyset, n \in \Nb\}$ is an open cover of $\eqv{\partiali\Cc}$. We will show that is a basis for the topology. Let $w \in W \subset \eqv{\partiali\Cc}$ where $W$ is an open set. Then
$\pi^{-1}(w)\subset \pi^{-1}(W)$. If $w' \in \pi^{-1}(w)$, there exists $U_{w'} \in \Uc$ such that $w \in U_{w'} \subset q^{-1}(W)$. Since $\pi^{-1}(w)$ is compact and $\Uc$ is closed under finite union, there exists $U \in \Uc$ such that 
$$
\pi^{-1}(w) \subset U \subset \pi^{-1}(W).
$$ 
Then $\pi^{-1}(w) \subset U^* \subset U \subset \pi^{-1}(W)$. Hence $w \in \pi(U^*) \subset W$ where  $\pi(U^*) \in \Vc$. This finishes the proof.

\subsection{Proof of Proposition~\ref{prop:convergence_grp}} We also note that similar arguments appear in~\cite[Theorem 10.3]{Choi_book} and~\cite[Proposition 8.8]{W2020}.  

Notice that Condition (1) implies that 
$$
\partiali \Cc \cap ~\overline{F_\Omega(x')} \subset [x']
$$
for all $x' \in \partiali\Cc$.

Suppose that $\{\gamma_n\}$ is a sequence of distinct elements in $\Gamma$. Fix $p_0 \in \Cc$. Passing to a subsequence we can suppose that $\gamma_n(p_0) \rightarrow x \in \partiali\Cc$, $\gamma_n^{-1}(p_0) \rightarrow y \in \partiali\Cc$, and $\gamma_n \rightarrow S$ in $\Pb(\End(\Rb^d))$.  

Then
$$
\lim_{n \rightarrow \infty} \gamma_n(z) = S(z)
$$ 
 for all $z \in \Pb(\Rb^d)- \Pb(\ker S)$ and the convergence is locally uniform. By Proposition~\ref{prop:dynamics_of_automorphisms_1}, 
$$
S\left( \overline{\Omega} - \Pb(\ker S) \right) \subset \overline{\Omega} \cap \Spanset F_\Omega(x) = \overline{F_\Omega(x)}.
$$
Then
\begin{align*}
S\left( \overline{\Cc} - \Pb(\ker S) \right) \subset \partiali \Cc \cap ~\overline{F_\Omega(x)} \subset [x].
\end{align*}
Proposition~\ref{prop:dynamics_of_automorphisms_1} also implies that $\Pb(\ker S) \cap \Omega = \emptyset$ and $y \in \Pb(\ker S)$. So if $z \in \partiali\Cc \cap \Pb(\ker S)$, then $[y,z] \subset \partial \Omega$ and then Condition (1) implies that $z \in [y]$. Hence
$$
\gamma_n|_{(\eqv{\partiali\Cc})-[y]}
$$ 
converges locally uniformly to the constant map $[x]$.

\subsection{Proof of Proposition~\ref{prop:conical_limit_pts}} We also note that similar arguments appear in~\cite[Theorem 10.3]{Choi_book} and~\cite[Proposition 8.17]{W2020}.

We start with a lemma.

\begin{lemma} If $y \in \partiali\Cc$ and $z \in F_\Omega( \partiali \Cc_y)\cap \partiali \Cc$, then $z\sim y$. 
\end{lemma}

\begin{proof} 
By hypothesis, there exists some $z^\prime \in \partiali \Cc_y$ with $z \in F_\Omega(z^\prime)$.

Since $[y]$ is closed, the extreme points of $\CH_{\Omega}([y])$ are contained in $[y]$. Then there exist $y_1,\dots,y_m \in [y]$ such that $$z' \in \relint \left( \CH_{\Omega}(\{y_1,\dots,y_m\}) \right) .$$ 
Thus $y_1, \dots, y_m \in \overline{F_{\Omega}(z')}$.  Then $[z,y_1] \subset \partial \Omega$ where $z,y_1 \in \partiali \Cc$. Then  Condition (1) implies that $z \sim y_1$. Thus $z \sim y$.
 \end{proof}

Now suppose $x \in \partiali\Cc$ is a uniform conical limit point of $\Gamma$ acting on $\Omega$. Fix $p_0 \in \Cc$ and a sequence $\{p_n\}$ in $[p_0, x)$ with $p_n \rightarrow x$. 

By Condition (2), there exist $r, D > 0$ such that: if $y_1 \not\sim y_2$, then 
$$
\diam_\Omega\Big( \Nc_\Omega\left(\Cc_{y_1},r\right) \cap \Nc_\Omega\left(\Cc_{y_2},r\right) \Big)  < D.
$$

\begin{lemma} For every $n \in \Nb$, there exists $q_n \in [p_n, x)$ such that 
$$
\Bc_\Omega(q_n, 2D) \cap (p_0, x) \not\subset \Nc_\Omega\left(\Cc_{y},r\right)
$$ 
for all $y \in \partiali \Cc$. 
\end{lemma} 

\begin{proof} 
Fix $n \in \Nb$ and suppose not. Then for every $q \in [p_n,x)$ there exists $y(q) \in \partiali\Cc$ such that 
$$
\Bc_\Omega(q, 2D) \cap (p_0, x) \subset \Nc_\Omega\left(\Cc_{y(q)},r\right)
$$ 
 (note that this implies that $\Cc_{y(q)}$ is a non-empty set).  We claim that $\Cc_{y(q)}$ does not depend on $q$. To show this it is enough to fix $q', q'' \in [p_n,x)$ with $\dist_\Omega(q', q'') \leq D$ and show that $\Cc_{y(q')} = \Cc_{y(q'')}$. In this case, $\Bc_\Omega(q', D) \subset  \Bc_\Omega(q'', 2D)$ and hence 
$$
\Bc_\Omega(q',D) \cap (p_0,x) \subset \Nc_\Omega\left(\Cc_{y(q')},r\right) \cap \Nc_\Omega\left(\Cc_{y(q'')},r\right).
$$
Then by our choice of $D$, we have $\Cc_{y(q')}=\Cc_{y(q'')}$. Thus $\Cc_{y(q)}$ does not depend on $q \in [p_n,x)$. Thus 
$$
(p_n,x) \subset  \Nc_\Omega\left(\Cc_{y(p_n)},r\right)
$$
Then Proposition~\ref{prop:asymp_sequences} implies that $x \in F_\Omega( \partiali\Cc_{y(p_n)})$. So by the previous lemma $x \sim y(p_n)$. Then $\Cc_x = \Cc_{y(p_n)}$ is non-empty, which contradicts the assumption that $[x] \subset F_\Omega(x)$. 
\end{proof} 

Now we finish the proof of Proposition \ref{prop:conical_limit_pts}.  Fix a sequence $\{q_n\}$ as in the above lemma and a sequence $\{\gamma_n\}$ in $\Gamma$ such that $\{\gamma_n(q_n)\}$ is relatively compact in $\Omega$. Passing to a subsequence we can suppose that $\gamma_n(p_0) \rightarrow b$, $\gamma_n^{-1}(p_0) \rightarrow c$, and $\gamma_n(x) \rightarrow a$. By the proof of Proposition~\ref{prop:convergence_grp}
$$
\gamma_n|_{(\eqv{\partiali\Cc})-[c]}
$$ 
converges locally uniformly to the constant map $[b]$. To show that $[x]$ is a conical limit point in the convergence group sense, we need to prove that $[x] = [c]$ and $[a] \neq [b]$. 

Suppose for a contradiction that $[a] = [b]$. Since $\{\gamma_n(q_n)\}$ is relatively compact in $\Omega$, we have $(a,b) \subset \Omega$. Hence $(a,b) \subset \Cc_a$ and so for $n$ large 
$$
\Bc_\Omega(\gamma_n(q_n), 2D) \cap (\gamma_n(p_0), \gamma_n(x)) \subset \Nc_\Omega\left(\Cc_{a},r\right).
$$ 
This implies that 
$$
\Bc_\Omega(q_n, 2D) \cap (p_0, x) \subset \Nc_\Omega\left(\Cc_{\gamma_n^{-1}(a)},r\right)
$$ 
which is a contradiction. So $[a] \neq [b]$. 

Since $\gamma_n([x]) \rightarrow [a]$ and $[a]\neq [b]$, we must have $[x]=[c]$.

\section{Proof of Theorem \ref{thm:main}}
\label{sec:proof_main}

In this section we prove Theorem~\ref{thm:main}. Suppose that $(\Omega, \Cc, \Gamma)$ is a naive convex co-compact triple. 

\medskip

\noindent \textbf{(1) $\Rightarrow$ (2): } Suppose that $\Gamma$ is relatively hyperbolic with respect $\Pc:=\{P_1, \ldots, P_m\}$ and 
$$
X_j := \CH_\Omega\Big( \limset(P_j) \cap \partiali \Cc\Big)\cap \Omega \quad j=1,\dots,m.
$$
Then Proposition \ref{prop:rh_implies_isolation} implies that $(\Cc, \dist_\Omega)$ is relatively hyperbolic with respect to 
$$
\Xc=\Gamma \cdot \{X_1,\dots, X_m\}.
$$ 
Then Proposition~\ref{prop:immediate_consequences_of_DS05} implies that $\Xc$ is a peripheral family of $(\Omega, \Cc, \Gamma)$.

\medskip

\noindent \textbf{(2) $\Rightarrow$ (1): } Suppose that $\Xc$ is a peripheral family of $(\Omega, \Cc, \Gamma)$ and $\Pc$ is a set of representatives of the $\Gamma$-conjugacy classes in $\{ \Stab_{\Gamma}(X) : X \in \Xc\}$.  Let 
$$ 
\qcc: = \partiali\Cc /{\sim}
$$ 
be as in Definition~\ref{defn:boundary quotient}. We claim that the action of $\Gamma$ on $\qcc$ satisfies Theorem~\ref{thm:yaman_criteria}.

\begin{lemma} The set $R:=\{(x,y) \in \partiali \Cc \times \partiali\Cc : x \sim y\}$ is closed. \end{lemma}

\begin{proof} 
Suppose that $\{ (x_n,y_n)\}$ is a sequence in $R$ converging to $(x,y)$ in $\partiali \Cc \times \partiali\Cc$. 

\medskip

\noindent \emph{Case 1:} Assume $(x,y) \subset \Omega$. Then $(x_n, y_n) \subset \Omega$ for $n$ sufficiently large and for such $n$ there exists $X_n \in \Xc$ with $x_n, y_n \in F_\Omega(\partiali X_n)$. By Theorem \ref{thm:boundary_faces} part (5) there is some $R > 0$ such that: for every $n$  there exist $x_n^\prime, y_n^\prime \in \partiali X_n$ where
$$
\dist_{F_\Omega(x_n)}(x_n, x_n^\prime) \leq R \quad \text{and} \quad \dist_{F_\Omega(y_n)}(y_n, y_n^\prime) \leq R.
$$
Then by Proposition~\ref{prop:Hausdorff_distance_between_lines} 
$$
\dist_{\Omega}^{\Haus}\Big( (x_n,y_n), (x_n^\prime,y_n^\prime) \Big) \leq R.
$$
Since $(x_n, y_n) \rightarrow (x,y)$, this implies that 
$$
\lim_{n,m \rightarrow \infty} \diam_\Omega \left( \Nc_\Omega( X_n, R+1) \cap \Nc_\Omega( X_m, R+1)\right) = \infty.
$$
Since $\Xc$ is strongly isolated, $X:=X_n = X_m$ for $n,m$ sufficiently large. Then 
$$
(x,y) \subset \Nc_\Omega(X, R)
$$
and so $x,y \in F_\Omega(\partiali X)$ by Proposition~\ref{prop:asymp_sequences}. Thus $x \sim y$.

\medskip

\noindent \emph{Case 2:} Assume $[x,y] \subset \partial\Omega$. Then $x \sim y$ by Theorem \ref{thm:boundary_faces} part (7).
\end{proof} 

Let $B = \left\{ [F_\Omega(\partiali X) \cap \partiali \Cc] : X \in \Xc\right\} \subset \qcc$. 

\begin{lemma} \ 
\begin{enumerate} 
\item $\qcc$ is a compact Hausdorff metrizable space. 
\item $\Gamma$ acts as a convergence group on $\qcc$. 
\item If $z \in \qcc - B$, then $z$ is a conical limit point of $\Gamma$.
\end{enumerate}
\end{lemma} 

\begin{proof} We use the results of Section~\ref{sec:bdry_quotient}. Part (1) follows from the previous lemma and Proposition \ref{prop:metrizable}.

 Notice that  Theorem \ref{thm:boundary_faces} part (7) implies that the quotient $\qcc$ satisfies Condition (1) from Section~\ref{sec:bdry_quotient}. We claim that the quotient $\qcc$ satisfies Condition (2) from Section~\ref{sec:bdry_quotient}. If $z \in \qcc - B$, then $\Cc_z = \emptyset$.  If $b = [F_\Omega(\partiali X) \cap \partiali \Cc] \in B$, then by Theorem \ref{thm:boundary_faces} parts (2) and (8), $\Stab_\Gamma(X)$ acts co-compactly on both $\Cc_b$ and $X$. So $\Cc_b$ is contained in a bounded neighborhood of $X$. By Theorem \ref{thm:boundary_faces} part (1), we can choose this bound to be independent of $b$. The since $\Xc$ is strongly isolated, $\qcc$ satisfies Condition (2). 
 
Then $\Gamma$ acts as a convergence group on $\qcc$ by Proposition~\ref{prop:convergence_grp}. If $z \in \qcc - B$, then $z = [F_\Omega(x) \cap \partiali \Cc]$ for some $x \in \partiali\Cc$ and so Proposition~\ref{prop:conical_limit_pts} implies that $z$ is a conical limit point of $\Gamma$. This proves parts (2) and (3).
\end{proof}

\begin{lemma} \ \begin{enumerate}
\item $\Gamma$ has finitely many orbits in $B$.
\item If $b=[F_\Omega(\partiali X) \cap \partiali \Cc] \in B$, then $\Stab_{\Gamma}(b) = \Stab_{\Gamma}(X)$. In particular, $\Stab_\Gamma(b)$ is finitely generated.
\item If $b \in B$, then $b$ is a bounded parabolic fixed point. 
\end{enumerate}
 \end{lemma} 
 
 \begin{proof} (1): This follows immediately from Theorem \ref{thm:boundary_faces} part (1). 
 
 (2): It is clear that $\Stab_{\Gamma}(b) \supset \Stab_{\Gamma}(X)$. For the other inclusion, if $\gamma \in \Stab_{\Gamma}(b)$, then 
 $$
 F_\Omega(  \partiali (\gamma X)) = \gamma F_\Omega(\partiali X) = F_\Omega(\partiali X).
 $$
 So $\gamma \in \Stab_{\Gamma}(X)$ by Theorem \ref{thm:boundary_faces} part (3). Thus $\Stab_{\Gamma}(b) = \Stab_{\Gamma}(X)$. The ``in particular'' part then follows from Theorem \ref{thm:boundary_faces} part (2). 
 
 (3): Fix $b = [F_\Omega(\partiali X) \cap \partiali \Cc] \in B$. By Theorem \ref{thm:boundary_faces} part (2), $\Stab_{\Gamma}(X)$ acts co-compactly on $X$. And by Theorem~\ref{thm:boundary_faces} part (6), if $x \in \partiali X$ and $[x,y] \subset \partiali \Cc$, then $y \in F_\Omega(\partiali X)$. 

 Then Theorem~\ref{thm:co_compact} implies that $\Stab_{\Gamma}(X)$ acts co-compactly on $\partiali \Cc - F_\Omega(\partiali X)$. Thus $\Stab_{\Gamma}(b)$ acts co-compactly on $\qcc-\{b\}$.
 \end{proof} 
 
 Then  Theorem~\ref{thm:yaman_criteria} implies that $\Gamma$ is relatively hyperbolic with respect to $\Pc$ and there exists a $\Gamma$-equivariant homeomorphism $\partial(\Gamma, \Pc) \rightarrow \qcc$. 

\medskip

\noindent \textbf{The ``moreover'' parts:} Now suppose that at least one of the conditions is satisfied. Proposition \ref{prop:rh_implies_isolation} implies parts (a) and (c). Part (b) was established in the proof that (2) $\Rightarrow$ (1) (see the last paragraph). Parts (d) and (e) follow from Theorem \ref{thm:boundary_faces}.

\section{The convex co-compact case}
\label{sec:cc}

In this section we prove Theorem \ref{thm:main_cc} and Proposition~\ref{prop: peripheral family in cc case}. As mentioned in the introduction, the key result which allows our results to simplify in the convex co-compact case is the following observation, see \cite{DGF2017}.

\begin{observation}\label{obs: contains faces}  
Suppose that $\Omega \subset \Pb(\Rb^d)$ is a properly convex domain and $\Gamma \leq \Aut(\Omega)$ is a convex co-compact subgroup. 
\begin{enumerate}
\item If $x \in \partial \Omega$ and $F_\Omega(x) \cap \partiali\Cc_\Omega(\Gamma) \neq \emptyset$, then $F_\Omega(x) \subset \partiali\Cc_\Omega(\Gamma)$. 
\item $\Lc_\Omega(\Gamma)=\partiali\Cc_\Omega(\Gamma)$. In particular, $\Lc_\Omega(\Gamma)$ is closed.
\end{enumerate} 
\end{observation}

\begin{proof}  For (1), fix $x \in \partial \Omega$ with $F_\Omega(x) \cap \partiali\Cc_\Omega(\Gamma) \neq \emptyset$. Then Proposition~\ref{prop:dynamics_of_automorphisms_2} implies that $F_\Omega(x) \subset \Lc_\Omega(\Gamma)$ and we have $\Lc_\Omega(\Gamma) \subset \partiali\Cc_\Omega(\Gamma)$ by the definition of $\Cc_\Omega(\Gamma)$. 

For (2), notice that the first line of this proof implies that $\partiali\Cc_\Omega(\Gamma) \subset \Lc_\Omega(\Gamma)$ and the reverse inclusion is by definition of $\Cc_\Omega(\Gamma)$.
\end{proof}

\subsection{Proof of Theorem~\ref{thm:main_cc}} Suppose that $\Gamma \leq \Aut(\Omega)$ is a convex co-compact group such that $\Gamma$ is relatively hyperbolic with respect to $\Pc:=\{P_1, \ldots, P_m\}$. Then by definition $(\Omega, \core(\Gamma),\Gamma)$ is a naive convex co-compact triple. Let $\Cc = \core(\Gamma)$ and 
$$
\Xc:=\Gamma \cdot \{\core(P_1),\dots, \core(P_m)\}.
$$

\begin{lemma}[Parts (a) and (c)] Each $P_j$ is a convex co-compact subgroup of $\Aut(\Omega)$ and $(\Cc,\hil)$ is relatively hyperbolic with respect to $\Xc$. 
\end{lemma}

\begin{proof}
Since $\limset(P_j) \subset \limset(\Gamma) \subset \partiali\Cc$, 
$$
\core(P_j)=\CH_\Omega (\limset(P_j)) \cap \Omega = \CH_\Omega(\limset(P_j) \cap \partiali \Cc) \cap \Omega. 
$$ 
Then Proposition~\ref{prop:rh_implies_isolation} implies that $P_j$ acts co-compactly on $\core(P_j)$. Then $(\Cc,\hil)$ is relatively hyperbolic with respect to $\Xc$. 
\end{proof}

If $X \in \Xc$, then  $X=\gamma\core(P_j)=\core(\gamma P_j \gamma^{-1})$ for some $\gamma \in \Gamma$ and $1 \leq j \leq m$.  Then, since $P_j$ is a convex co-compact subgroup of $\Aut(\Omega)$, Observation \ref{obs: contains faces} implies that
\begin{align}
\label{eqn:face-x}
F_\Omega(\partiali X) = \partiali X = \Lc_\Omega(\gamma P_j \gamma^{-1}).
\end{align}

\begin{lemma}[Part (e)]
\label{lem: line segements are in Pj limit sets} If $\ell \subset \partiali\core(\Gamma)$ is a non-trivial line segment, then $\ell \subset\Lc_\Omega( \gamma P_j \gamma^{-1})$ for some $\gamma \in \Gamma$ and $P_j \in \Pc$.
\end{lemma} 

\begin{proof} Fix $x \in \relint(\ell)$. Then $\dim F_\Omega(x) \geq 1$ and so, by Observation \ref{obs: contains faces},  
$$
\diam_{F_\Omega(x)} \left( F_\Omega(x) \cap \partiali\Cc \right) = \diam_{F_\Omega(x)} F_\Omega(x) = \infty.
$$
Thus by part (d) of Theorem~\ref{thm:main} there exist $\gamma \in \Gamma$ and $P_j \in \Pc$ such that
$$
x \in F_\Omega\left( \partiali X \right)
$$
where $X=\core(\gamma P_j \gamma^{-1}) = \gamma \core(P_j)$. Then Equation \eqref{eqn:face-x} implies that
\begin{align*}
\relint(\ell) \subset F_{\Omega}(x) \subset  F_{\Omega}(\partiali X) = \limset( \gamma P_j \gamma^{-1}).
\end{align*}
Since $\limset( \gamma P_j \gamma^{-1})$ is closed, see Observation \ref{obs: contains faces},  this completes the proof. 
\end{proof} 

Recall that, by definition, $[\partiali\core(\Gamma)]_{\Pc}$ is obtained from $\partiali \core(\Gamma)$ by collapsing each $\Lc_{\Omega}(\gamma P_j \gamma^{-1})$ to a point. Hence Equation \eqref{eqn:face-x} and Lemma~\ref{lem: line segements are in Pj limit sets} imply that 
\begin{equation}
\label{eqn:equality_of_quotients}
[ \partiali \Cc]_{\Xc} = [\partiali \Cc]_{\Pc}.
\end{equation}

\begin{lemma}[Part (b)]
There is a $\Gamma$-equivariant homeomorphism 
$$
\partial(\Gamma,\Pc) \rightarrow [\partiali\core(\Gamma)]_{\Pc}.
$$
\end{lemma}
\begin{proof}
Follows immediately from Theorem \ref{thm:main} part (b) and Equation \eqref{eqn:equality_of_quotients}.
\end{proof}

To prove part (d) of Theorem~\ref{thm:main_cc} we will use the following lemma. 

\begin{lemma} [{\cite[Lemma 15.5]{IZ2019b}}]
\label{lem:scaling_near_non_C1_point}
Assume that $x \in \partiali \Cc$ is not a $C^1$-smooth point of $\partial \Omega$ and $q \in \Cc$. For any $r > 0$ and $\epsilon > 0$ there exists $q_{r,\epsilon} \in (x,q]$ with the following property: if $p \in  (x,q_{r,\epsilon}]$, then there exists a properly embedded simplex $S=S(p) \subset \Cc$ of dimension at least two such that 
\begin{align}
\label{eq:ball_near_simplex_3}
\Bc_{\Omega}(p,r) \cap (x,q] \subset \Nc_\Omega(S,\epsilon).
\end{align}
\end{lemma}

Now the proof of part (d).

\begin{lemma}[Part(d)] 
If $x \in \partiali\core(\Gamma)$ is not a $\Cc^1$-smooth point of $\Omega$, then $x \in \Lc_\Omega(\gamma P_j\gamma^{-1})$ for some $\gamma \in \Gamma$ and $P_j \in \Pc$. 
\end{lemma}

\begin{proof} Fix $q \in \Cc$. By Theorem~\ref{thm:main}, $\Xc$ is a peripheral family of $(\Omega, \Cc, \Gamma)$. So there exist constants $D_1, D_2 > 0$ such that:
\begin{enumerate}[label=(\alph*)]
\item If $S \subset \Cc$ is a properly embedded simplex of dimension at least two, then there exists $X \in \Xc$ with $S \subset \Nc_\Omega(X,D_2)$. 
\item If $X_1, X_2 \in \Xc$ are distinct, then
\begin{align*}
\diam_\Omega \left( \Nc_\Omega(X_1,D_2+1) \cap \Nc_\Omega(X_2,D_2+1) \right) < D_1.
\end{align*}
\end{enumerate}

Let $q^\prime \in (x,q]$ satisfy Lemma~\ref{lem:scaling_near_non_C1_point} for the constants $r=D_1+1$ and $\epsilon = 1$. Fix a sequence $\{p_n\}$ in $(x,q^\prime]$ such that $p_n \rightarrow x$ and $\dist_\Omega(p_n,p_{n+1}) \leq 1$.  For each $n \geq 1$, fix $S_n \subset \Cc$ a properly embedded simplex of dimension at least two such that 
\begin{align*}
\Bc_{\Omega}(p_n,D_1+1) \cap (x,q] \subset \Nc_\Omega(S_n,1).
\end{align*}
Then fix $X_n \in \Xc$ such that 
$$
S_n \subset \Nc_\Omega(X_n, D_2).
$$
Then
\begin{align*}
\Bc_\Omega(p_n, D_1) & \cap (x,q] \subset \Big(\Bc_{\Omega}(p_n,D_1+1) \cap \Bc_\Omega(p_{n+1},D_1+1) \Big)\cap (x,q]  \\
& \subset \Nc_\Omega(X_n,D_2+1) \cap \Nc_\Omega(X_{n+1},D_2+1).
\end{align*} 
Since $\diam_\Omega( \Bc_\Omega(p_n, D_1) \cap (x,q] ) \geq D_1$, this implies that there exists $X \in \Xc$ such that $X=X_n$ for all $n$ large enough. Then 
$$
(x,q^\prime] \subset \Nc_\Omega(X, D_2+1).
$$

Now $X = \core(\gamma P_j \gamma^{-1})$ for some $\gamma \in \Gamma$ and $P_j \in \Pc$. Finally, Proposition~\ref{prop:asymp_sequences} and Equation \eqref{eqn:face-x} implies that $x \in \partiali \core(\gamma P_j \gamma^{-1}) = \limset( \gamma P_j \gamma^{-1})$. 
\end{proof}

\subsection{Proof of Proposition~\ref{prop: peripheral family in cc case}} Suppose that $\Gamma \leq\Aut(\Omega)$ is a convex co-compact subgroup and $\Xc$ is a $\Gamma$-invariant collection of closed unbounded convex subsets of $\Omega$.

\medskip

\noindent \textbf{(1) $\Rightarrow$ (2):} Suppose that $\Xc$ is a peripheral family of $(\Omega, \core(\Gamma), \Gamma)$.

\begin{lemma} $\Xc$ is closed in the local Hausdorff topology induced by the Hilbert metric $\dist_\Omega$. \end{lemma}

\begin{proof} This follows immediately from Lemma~\ref{lem:finitely_many_simplices_in_compact_set}.\end{proof} 

\begin{lemma}  If $X_1, X_2 \in \Xc$ are distinct, then $\partiali X_1 \cap \partiali X_2 = \emptyset$. \end{lemma}

\begin{proof}We prove the contrapositive. Suppose $x \in \partiali X_1 \cap \partiali X_2$. Fix $q_1 \in X_1$ and $q_2 \in X_2$. Then 
$$
\dist_{\Omega}^{\Haus}( (x,q_1], (x,q_2]) \leq \dist_\Omega(q_1,q_2)
$$
by Proposition~\ref{prop:Hausdorff_distance_between_lines}. So
$$
(x,q_1] \subset \Nc_\Omega( X_1, \dist_\Omega(q_1,q_2)+1) \cap \Nc_\Omega( X_2, \dist_\Omega(q_1,q_2)+1)
$$
and hence $X_1 = X_2$. 
\end{proof}

\begin{lemma}If $\ell \subset \partiali\core(\Gamma)$ is a non-trivial line segment, then $\ell \subset \partiali X$ for some $X \in \Xc$. 
\end{lemma}

\begin{proof}Fix $x \in \relint(\ell)$. Then $\dim F_\Omega(x) \geq 1$ and so by Observation~\ref{obs: contains faces}, 
$$
\diam_{F_\Omega(x)} \left( F_\Omega(x) \cap \partiali\Cc \right) = \diam_{F_\Omega(x)} F_\Omega(x) = \infty.
$$
Thus by Theorem~\ref{thm:boundary_faces} part (4) there exists $X \in \Xc$ with $x \in F_\Omega(\partiali X)$. We claim that $F_\Omega(x) \subset \partiali X$ which will imply the lemma.

To show that $F_\Omega(x) \subset \partiali X$, it suffices to fix an extreme point $e \in \partial F_\Omega(x)$ of $F_\Omega(x)$ and show that $\partiali X$ contains $e$. By Observation~\ref{obs: contains faces} and Theorem~\ref{thm:boundary_faces} part (5) there exists $R > 0$ such that 
$$
\dist_{F_\Omega(x)}^{\Haus}( F_\Omega(x), \partiali X \cap F_\Omega(x)) \leq R. 
$$
Since $F_{F_\Omega(x)}(e) = \{e\}$, then Proposition~\ref{prop:asymp_sequences} implies that $e \in \partiali X$. 
\end{proof}  

\noindent \textbf{(2) $\Rightarrow$ (1):} Suppose that $\Xc$ has the following properties: 
\begin{enumerate}[label=(\alph*)]
\item $\Xc$ is closed in the local Hausdorff topology induced by the Hilbert metric $\dist_\Omega$. 
\item If $X_1, X_2 \in \Xc$ are distinct, then $\partiali X_1 \cap \partiali X_2 = \emptyset$. 
\item If $\ell \subset \partiali\core(\Gamma)$ is a non-trivial line segment, then $\ell \subset \partiali X$ for some $X \in \Xc$. 
\end{enumerate} 

\begin{lemma}\label{lem : disjoint faces in endgame} If $X_1, X_2 \in \Xc$ and $F_\Omega(\partiali X_1) \cap F_\Omega(\partiali X_2) \neq \emptyset$, then $X_1 = X_2$. \end{lemma}

\begin{proof} By hypothesis  there exists $x_1 \in \partiali X_1$ and $x_2 \in \partiali X_2$ with $F_\Omega(x_1) = F_\Omega(x_2)$. If $x_1 = x_2$, then property (b) implies that $X_1 = X_2$. Otherwise, $[x_1,x_2] \subset F_\Omega(x_1) \subset \partial \Omega$ and so by property (c) there exists $X_3$ with $[x_1,x_2] \subset \partiali X_3$. But then by property (b), $X_1 = X_3 = X_2$. 
\end{proof}

\begin{lemma}\label{lem:discrete in endgame} $\Xc$ is discrete in the local Hausdorff topology induced by the Hilbert metric $\dist_\Omega$. \end{lemma} 

\begin{proof} Fix a sequence $\{X_n\}$ in $\Xc$ converging to some closed subset $X$. Since $\Xc$ is closed, $X \in \Xc$. Fix $p \in X$ and $x \in \partiali X$. 

We claim that $X_n = X$ when $n$ is sufficiently large. Suppose not. Then after passing to a subsequence we can suppose that $X_n \neq X$ for all $n$. For each $n$ fix $p_n \in X_n$ such that $p_n \rightarrow p$. Passing to a tail of our sequence we can suppose that $\dist_\Omega(p_n, p) < 1$ for all $n$. The previous lemma implies that $x \notin F_\Omega(\partiali X_n)$ and so by  Proposition~\ref{prop:asymp_sequences}
$$
\lim_{q \in [p_n,x), q \rightarrow x} \dist_\Omega(q, X_n) = \infty.
$$
So for each $n$ there exists $q_n \in [p_n, x)$ with $\dist_\Omega(q_n, X_n) = 1$. Notice that 
$$
\lim_{n \rightarrow \infty} \dist_\Omega(q_n, X) \leq \lim_{n \rightarrow \infty} \dist_\Omega^{\Haus}\left( (x,p_n], (x,p]\right) \leq \lim_{n \rightarrow \infty} \dist_\Omega(p_n, p) =0
$$
by Proposition~\ref{prop:Hausdorff_distance_between_lines}. Also, since $X_n \rightarrow X$ in the local Hausdorff topology, we must have $q_n \rightarrow x$. 

Since $q_n \in \Cc$, there exists $\{\gamma_n\}$ in  $\Gamma$ such that $\{\gamma_n(q_n)\}$ is relatively compact in $\Omega$. Then passing to a subsequence we can suppose that $\gamma_n(q_n) \rightarrow q_\infty \in \Cc$, $\gamma_n(p) \rightarrow p_\infty \in \overline{\Cc}$, $\gamma_n(X) \rightarrow X_\infty$, and $\gamma_n(X_n) \rightarrow Y_\infty$. Since $q_n \rightarrow x$, we must have $p_\infty \in \partiali \Cc$. Since $\Xc$ is $\Gamma$-invariant and closed in the local Hausdorff topology, $X_\infty, Y_\infty \in \Xc$. 

By construction $q_\infty \in X_\infty$ and $\dist_\Omega(q_\infty,Y_\infty)=1$. So $X_\infty \neq Y_\infty$. Also
$$
\lim_{n \rightarrow \infty} \dist_\Omega( \gamma_n(p), \gamma_n(p_n)) = \lim_{n \rightarrow \infty} \dist_\Omega(p_n, p) =0,
$$
so  Proposition \ref{prop:asymp_sequences} implies that $\gamma_n(p_n) \rightarrow p_\infty$. Then, by Proposition~\ref{prop:Hausdorff_distance_between_lines}, 
$$
[q_\infty, p_\infty) \subset \Nc_\Omega( Y_\infty, 2)
$$
since $[q_n, p_n] \subset \Nc_\Omega(X_n, 2)$ for all $n$. So Proposition~\ref{prop:asymp_sequences} implies that $p_\infty \in F_\Omega(\partiali Y_\infty)$. However, by construction, $p_\infty \in \partiali X_\infty$ and so Lemma~\ref{lem : disjoint faces in endgame} implies that $X_\infty = Y_\infty$. So we have a contradiction. 
\end{proof}

\begin{lemma} $\Xc$ is strongly isolated.
\end{lemma}

\begin{proof} Fix $r > 0$ and suppose for a contradiction that for every $n \in \Nb$ there exist $X_n, Y_n \in \Xc$ distinct such that 
$$
\diam_\Omega\left(\Nc_\Omega(X_n,r) \cap \Nc_\Omega(Y_n, r) \right) \geq n. 
$$
Proposition~\ref{prop:Hausdorff_distance_between_lines} implies that $C_n:=\Nc_\Omega(X_n,r) \cap \Nc_\Omega(Y_n, r)$ is convex. 

We claim that each $C_n$ is bounded in $(\Cc, \dist_\Omega)$. If not, there exists $x \in \partiali C_n$. Then Proposition~\ref{prop:asymp_sequences}  implies that $x \in F_\Omega(\partiali X_n)  \cap F_\Omega(\partiali Y_n)$. So $X_n = Y_n$ by Lemma~\ref{lem : disjoint faces in endgame}. Thus we have a contradiction and so each $C_n$ is bounded in $(\Cc, \dist_\Omega)$.

Now for each $n$ let $(p_n, q_n) \subset C_n$ denote an open line segment with maximal length (with respect to the Hilbert metric). Then $\dist_\Omega(p_n, q_n) \geq n$.

Now fix a sequence $\{\gamma_n\}$ in $\Gamma$ such that $\{\gamma_n(p_n)\}$ is relatively compact in $\Omega$. Passing to a subsequence we can assume $\gamma_n(p_n) \rightarrow p$, $\gamma_n(q_n) \rightarrow q$, $\gamma_n(X_n) \rightarrow X$, and $\gamma_n(Y_n) \rightarrow Y$. Then $q \in \partiali \Cc$,
$$
(q,p) \subset \Nc_\Omega(X, r+1) \cap \Nc_\Omega(Y, r+1),
$$
and $X,Y \in \Xc$. Proposition~\ref{prop:asymp_sequences}  implies that $q \in F_\Omega(\partiali X) \cap F_\Omega( \partiali Y)$. So  Lemma~\ref{lem : disjoint faces in endgame} implies that $X=Y$. Then Lemma~\ref{lem:discrete in endgame} implies that $\gamma_n(X_n) = X = Y = \gamma_n(Y_n)$ for $n$ large. So $X_n = Y_n$ for $n$ large and we have a contradiction. 
 \end{proof}
 
 The following result implies that $\Xc$ coarsely contains the properly embedded simplices of $\Cc$.
 
 \begin{lemma} If $S \subset \Cc$ is a properly embedded simplex with dimension at least two, then there exists $X \in \Xc$ with $S \subset X$.
\end{lemma}

\begin{proof} Fix a properly embedded simplex $S \subset \Cc$ with dimension at least two. Let $v_1,\dots, v_k$ denote the vertices of $S$. By property (c), for each $2 \leq j \leq k$ there exists some $X_j \in \Xc$ with $[v_1,v_j] \subset \partiali X_j$. Then by property (b), $X_2=X_3=\cdots = X_k$. So by convexity $S \subset X_2$. 
\end{proof}

\medskip

\noindent \textbf{The ``moreover'' part in Proposition \ref{prop: peripheral family in cc case}.} Suppose that (1) and (2) hold, and let $\Pc$ be a set of representatives of the $\Gamma$-conjugacy classes in $\{ \Stab_{\Gamma}(X) : X \in \Xc\}$. To show that the quotients $[ \partiali \Cc]_{\Xc}$ and $[\partiali \Cc]_{\Pc}$ coincide it suffices to prove the following. 

\begin{lemma} If $X \in \Xc$, then $\partiali X = F_\Omega(\partiali X) = \Lc_\Omega( \Stab_{\Gamma}(X) )$. \end{lemma} 

\begin{proof} Let $P =  \Stab_{\Gamma}(X)$. 

We first observe that $\partiali X = F_\Omega(\partiali X)$. By definition $\partiali X \subset F_\Omega(\partiali X)$. For the other inclusion, fix $x \in F_\Omega(\partiali X)$. Then fix $x^\prime \in \partiali X$ with $x \in F_\Omega(x^\prime)$. Observation~\ref{obs: contains faces} implies that $x \in \partiali\Cc$. So, if $x^\prime \neq x$, then there exists $Y \in \Xc$ with $[x^\prime, x] \subset \partiali Y$. But then $\partiali X \cap \partiali Y \neq \emptyset$ which implies that $X = Y$. So $x \in \partiali X$. Thus $\partiali X = F_\Omega(\partiali X)$. 

Next we show that $\Lc_\Omega(P)$ is a subset of $\partiali X$. Fix $x \in \Lc_\Omega(P)$, then there exist $p \in \Omega$ and a sequence $\{g_n\}$ in $P$ such that $g_n(p) \rightarrow x$. Fix $q \in X$. By passing to a subsequence we may suppose that $g_n(q) \rightarrow x^\prime \in \partiali X$. Then Proposition~\ref{prop:asymp_sequences}  implies that $x \in F_\Omega(x^\prime) \subset \partiali X$. 

Finally, Theorem \ref{thm:boundary_faces} part (2) implies that $P$ acts co-compactly on $X$ and so $\partiali X \subset \Lc_\Omega(P)$ by Proposition~\ref{prop:dynamics_of_automorphisms_2}. 
\end{proof}

\bibliographystyle{alpha}
\bibliography{geom}

\end{document}